\documentclass[a4paper]{article}
\usepackage{fullpage}
\usepackage{xfrac}
\usepackage{textcomp}
\usepackage{amsfonts}
\usepackage{amsmath}
\usepackage{ntheorem}
\usepackage{amssymb}

\usepackage[hidelinks]{hyperref}
\usepackage{listings}
\usepackage{graphicx}
\let\inf\relax \DeclareMathOperator*\inf{\vphantom{p}inf}
 
\let\lim\relax \DeclareMathOperator*\lim{\vphantom{p}lim}
\let\max\relax \DeclareMathOperator*\max{\vphantom{p}max}
\newtheorem{theorem}{Theorem}%[section]

\newcounter{dummy} \numberwithin{dummy}{section}
\newtheorem{proposition}[dummy]{Proposition}
\newtheorem{lemma}[dummy]{Lemma}
\newtheorem{definition}[dummy]{Definition}

\newenvironment{proof}{\noindent\\ \noindent{\sc    Proof}}{{\samepage\par\nopagebreak\hbox to\hsize{\hfill$\Box$}}}

\newcommand{\be}{\begin{equation}}
\newcommand{\ee}{\end{equation}}
\addcontentsline{toc}{section}{Abstract}
\usepackage[nottoc,numbib]{tocbibind}

\numberwithin{equation}{section}
\newcommand{\rP}{\mathrm{P\space }} %probability, when spacing isn't a worry
\newcommand{\rE}{\mathrm{E\space }} %expectation, when spacing isn't a worry
\newcommand{\rV}{\mathrm{Var\hspace{0.2mm}}} %expectation, when spacing isn't a worry

\begin{document}

\title{Critical Galton-Watson processes with overlapping generations}

\author{ Serik Sagitov \\
Chalmers University of Technology and University of Gothenburg}

\date{} 
\maketitle

\begin{abstract}
A properly scaled critical Galton-Watson process converges to a continuous state critical branching process $\xi(\cdot)$  as the number of initial individuals tends to infinity.
 We extend this classical result by allowing for overlapping generations and considering a wide class of population counts.
The main result of the paper establishes a convergence of the finite dimensional distributions for a scaled vector of multiple population counts. The set of the limiting distributions is conveniently represented in terms of integrals $(\int_0^y\xi(y-u)du^\gamma, y\ge0)$ with a pertinent $\gamma\ge0$. 
\end{abstract}

\tableofcontents

\section{Introduction}\label{sec:int}

One of the basic stochastic population model of a self-reproducing system is build upon two assumptions
\begin{description}
\item[]  (A) different individuals live independently from each other according to the same individual life law described in (B),
\item[] (B) an individual dies at age one and at the moment of death gives birth to a random number $N$ of offspring.
\end{description}
Within this model, the numbers of individuals $Z_0, Z_1,\ldots$, born at times $t=0,1,\ldots$, form a Markov chain, whose transition probabilities are fully described by the distribution of the offspring number $N$. The Markov chain $\{Z_t, t\ge0\}$ is usually called a Galton-Watson process, or \textit{GW-process} for short.
A GW-process is classified as subcritical, critical, or supercritical, depending on whether the mean offspring number $\rE(N)$ is less than, equal to, or larger than the critical value 1.

%In view of  
%\[\rE(Z_t\,|\,Z_0)=Z_0\mu^t,\quad t=1,2,\ldots,\] 
%where $\mu=\rE(N)$ is the mean offspring number, there is a crucial distinction between three regimes of reproduction:  subcritical, critical, and supercritical, depending on whether $\mu<1$,  $\mu=1$, or $\mu>1$.
%In the subcritical case,  the Markov chain $Z_0, Z_1,\ldots$ gets quickly absorbed at zero.
% A supercritical Galton-Watson process grows exponentially, however, it may also experience a quick extinction. In the critical case, unless $\rV(N)=0$, extinction is inevitable, but it occurs much slower than in the subcritical case. 

%Due to the martingale property
%\[\rE(Z_{t+1}\,|\,Z_t)=Z_t,\quad t\to\infty, \]
%saying that the mean population size stays constant over time $\rE(Z_t\,|\,Z_0)=Z_0$. %Such a process starting from a large number $Z_0=n$ individuals, will stay alive for a long time proportional to  $n$. More precisely,
%If 
%\begin{equation}
%  \label{crit}
%  \mu=1,\quad 0<\sigma<\infty, 
%\end{equation}
%then it is known that for a fixed  $u\ge0$, there holds a convergence in distribution
%\begin{equation}
%  \label{cg1}
% (n^{-1}Z_{nu}\,|\,Z_0=n)\to \xi(u),\quad n\to \infty,
%\end{equation}
%where the limiting distribution has the log-Laplace transform
%\begin{equation}
%  \label{cg1}
%-\ln\rE(e^{-\lambda \xi(u)})= (\lambda^{-1}+\tfrac{u\sigma^2}{2} )^{-1},\quad \lambda \ge0,\quad u\ge0.
%\end{equation}
It is known, that in the critical case, with
\begin{equation}
  \label{crit}
  \rE(N)=1,\quad \rV(N)=2b, \quad b<\infty, 
\end{equation}
the  finite dimensional distributions (fdd's) of a properly  scaled GW-process converge
\begin{equation}
\label{cgw}
\{n^{-1}Z_{nu},\ u\ge0\,|\,Z_0=n\}\stackrel{\rm fdd\ }{\longrightarrow}\{\xi(u),\ u\ge0\,|\,\xi(0)=1\},\ n\to \infty,
\end{equation}
and the limiting fdd's are represented by a continuous state branching process $\xi(\cdot)$, which is a continuous time Markov process with a transition law determined by
\begin{align}
  \label{cbp}
  \rE\big(e^{-\lambda\xi(v+u)}\,|\,\xi(v)=x\big)=e^{-\frac{\lambda x}{1+\lambda bu}} \, ,\quad v, u, x, \lambda\ge0.
\end{align}
Note how the parameter $b$ acts as a time scale: the larger is the variance of $N$, the faster is changing the population size.

%This formula for the log-Laplace transform implies 
%\[ \rE(\xi(u)\,|\,\xi(0)=x)=x, \quad \rV(\xi(u)\,|\,\xi(0)=x)=2bxu.\]
%In particular, if $b=0$, the process in question is simply a constant.

In this paper, we study $\{Z(t), t\ge0\}$,  a Galton-Watson process with overlapping generations, or \textit{GWO-process} for short, where $Z(t)$ is the number of individuals alive at time $t$ in a reproduction system satisfying the following two assumptions 
 \begin{description}
\item[]  (A$^*$) different individuals live independently from each other according to the same individual life law described in (B$^*$),
\item[] (B$^*$) an individual lives $L$ units of time and gives $N$ births at random ages $\tau_1,\ldots,\tau_N$, satisfying
\begin{equation}
  \label{ir}
1\le\tau_1\le \ldots\le \tau_N\le L.
\end{equation}
\end{description}
Assumption (B$^*$) allows for overlapping generations, when mothers may coexist with their daughters.  %unless $\rP(L=1)=1$, when (B) and (B$^*$) are equivalent. 
We focus on the critical case \eqref{crit}  and aim at an extension of \eqref{cgw} to the GWO-processes. 

%\begin{figure}
%\begin{center}
% \includegraphics[width=8cm]{Over.pdf}
%  \caption{Four overlapping generations: first generation in red, second generation in green, third in blue, and fourth in black. For this particular realisation, we have the following consecutive numbers of newborns:\ \ $Z_0=2$, $Z_1=1$, $Z_2=0$, $Z_3=Z_4=3$, $Z_5=0$, $Z_6=1$, $Z_7=0$. At the same time, the numbers of individuals alive are $Z(0)=Z(1)=Z(2)=2$, $Z(3)=Z(4)=3$, $Z(5)=Z(6)=Z(7)=2$.
% } \label{fig}
%\end{center}
%\end{figure}
%

The process $\{Z(t), t\ge0\}$, being non-Markov in general, is studied with help of an associated renewal process, introduced in Section \ref{psm}.  The mean inter-arrival time 
\begin{align}
  \label{aha}
a:=\rE(\tau_1+\ldots+\tau_N)
 \end{align}
of this renewal process gives us the average generation length. 
 It is important to 
distinguish between the average generation length $a$, which in this paper will be assumed finite, and the average life length
$\mu:=\rE(L)$, allowed to be infinite.

With a more sophisticated reproduction mechanism \eqref{ir}, there are many interesting population counts to study, alongside the number of newborns $Z_t$ and the number of individual alive $Z(t)$ at the time $t$. (For GW-processes, $a=1$ and $Z(t)$ equals $Z_t$, since all alive individuals are newborn.) %Figure \ref{fig} illustrates the case of overlapping generations. 
An interesting case of population counts is treated by Theorem \ref{deco} dealing with decomposable multitype GW-processes. Theorem \ref{deco} is obtained as an application of the main results of the paper, Theorems \ref{nrt}, \ref{nrtg}, \ref{nrtm}, stated and proven in Section \ref{sec:lt}. The following three statements are straightforward corollaries of our Theorems \ref{nrt}, \ref{nrtg}, and \ref{nrtm} respectively. In these theorems, it is always assume that the GWO-process stems from a large number  $Z_0=n$ of progenitors born at time zero.  
\\

\noindent\textbf{Corollary 1}. 
\textit{
Consider a GWO-process satisfying (\ref{crit}) and $a<\infty$. If $\mu<\infty$, then }
 $$\{n^{-1}Z(nu),\ u>0\,|\,Z_0=n\}\stackrel{\rm fdd\ }{\longrightarrow}\{\mu a^{-1}\xi(ua^{-1}),\ u>0\,|\,\xi(0)=1\},\quad n\to\infty.$$

 \noindent\textbf{Corollary 2}. 
\textit{ Consider a GWO-process satisfying (\ref{crit}) and $a<\infty$. If $\mu=\infty$, and for some  slowly varying function at infinity}  $\mathcal{L}(\cdot)$,
\begin{equation}\label{heavy}
\sum\nolimits_{j=0}^t \rP(L>j)=t^\gamma\mathcal{L}(t),\quad 0\le\gamma\le1,\quad t\to\infty,
\end{equation}
\textit{ then, as}  $n\to\infty$,
$$\{n^{-1-\gamma}\mathcal{L}^{-1}(n)Z(nu),\ u>0\,|\,Z_0=n\}\stackrel{\rm fdd\ }{\longrightarrow}\{a^{\gamma-1}\xi_\gamma(ua^{-1}),\ u>0\,|\,\xi(0)=1\}.$$

\noindent\textbf{Corollary 3}. 
\textit{
Consider a GWO-process satisfying \eqref{crit}, $a<\infty$, and \eqref{heavy}. Then, as }$n\to\infty$,
 $$\{(n^{-1-\gamma}\mathcal{L}^{-1}(n)Z(nu),n^{-1}Z_{nu}),\ u>0\,|\,Z_0=n\}\stackrel{\rm fdd\ }{\longrightarrow}\{(a^{\gamma-1}\xi_\gamma(ua^{-1}), a^{-1}\xi(ua^{-1})),\ u>0\,|\,\xi(0)=1\}.$$

Notice that condition \eqref{heavy} holds even in the case $\mu<\infty$, with $\gamma=0$ and $\mathcal{L}(t)\to\mu$ as $t\to\infty$. The family of processes $\{\xi_\gamma(\cdot)\}_{\gamma\ge0}$ emerging in our limit theorems can be expressed in the integral form
\begin{align}
  \label{asmaa}
\xi_0(u):=\xi(u),\text{ for } \gamma=0, \text{ and }\xi_\gamma(u):=\int_0^u\xi(u-v)dv^\gamma \text{ for }\gamma>0, \quad u\ge0,
\end{align}
which  is treated as a convenient representation of the limiting fdd's, see Section \ref{secta}.

The following remarks comment on relevant literature and mention an interesting  open problem.
\begin{enumerate}
\item  The GW-process is a basic model of the biologically motivated  theory of branching processes, see \cite{AN}, \cite{HJV}. The critical GW-process can be viewed as a stochastic  model of a sustainable reproduction, when a mother produces on average one daughter, see \cite{Kim}.  
\item The GWO-process is a discrete time version of the so-called general branching process, often called the Crump-Mode-Jagers process, see \cite{HJV}, \cite{J},  \cite{JS1}, \cite{ZT}.
\item The fruitful concept of population counts, allowing for a variety of individual scores, see Section \ref{psm}, was first introduced in \cite{J74}. The interested reader may find several demographical examples of population counts  in \cite{J74} and \cite{J}.
\item Above mentioned Theorem  \ref{deco} deals with the decomposable critical multitype  GW-processes.
In a more general setting, such processes were studied in \cite{FN}, addressing related issues by applying a different approach.
\item Compared to earlier attempts, see  \cite{V86}, \cite{V89}, and especially \cite{90}, the current treatment of critical age-dependent branching  processes is made more accessible by restricting the analysis to the case of finite $\rV(N)$ and  $a$, as well as focussing on the discrete time setting.
\item Our proofs do not use \eqref{cgw} as a known fact (unlike for example \cite{H2}, addressing a related problem). Therefore, convergence \eqref{cgw} can be derived from  the  above mentioned Corollary 1.
\item The branching renewal approach, introduced in Section \ref{sec:bre}, takes its origin in \cite{G}. 
\item The idea of studying branching processes starting from a large number of individuals is quite old, see \cite{Sev1} and especially \cite{L}. For a most recent paper in the continuous time setting, see \cite{MM}.
\item The definitions and basic properties of slowly and regularly varying functions, used in this paper, can be found in \cite{F2}. We apply some basic facts of the renewal theory from \cite{F1}.
\item Our limit theorems are stated in terms of the fdd-convergence. Finding simple conditions on the individual scores, ensuring weak convergence in the Skorokhod sense, is an open problem.
\end{enumerate}

\subsection*{Notational agreements}

\begin{enumerate}
\item To avoid confusion, we set apart discrete and continuous variables:
$$i, j, k, l, n, p,q, s, t \in\mathbb Z=\{0,\pm1,\pm2,\ldots\},\qquad u, v, x, y, z, \lambda\in[0,\infty).$$
Mixed products are treated as integer numbers, so that $nu$ stands for $\lfloor nu\rfloor$. The latter results in $\tfrac{nu}{n}$ not always being equal to $u$.

\item We distinguish between  a stronger  and a weaker forms of the uniform convergence
$$f^{(n)}(y)\stackrel{y\,}{\Rightarrow}f(y),\qquad f^{(n)}(y)\stackrel{y}{\to}f(y),\quad n\to\infty,$$
which respectively require the relations
$$\sup_{0\le y\le y_1}\,|\,f^{(n)}(y)-f(y)\,|\,\to0,\qquad \sup_{y_0\le y\le y_1}\,|\,f^{(n)}(y)-f(y)\,|\,\to0,\quad n\to\infty,$$ 
to hold for any  $0<y_0<y_1<\infty$.

\item We will write
\[\rE_n(\cdot):=\rE(\cdot\,|\,Z_0=n)\]
to say that the expected value is computed under the assumption that the GWO-process starts from $n$ individuals born at time 0.
With a little risk of confusion, we will also write 
\[\rE_x(\cdot):=\rE(\cdot\,|\,\xi(0)=x),\]
when the expectation deals with the finite dimensional distributions of the continuous state branching process $\xi(\cdot)$.
\item We will often use the following two shortenings
\[e_1^x:=1-e^{-x},\qquad e_2^x:=x-e_1^x=e^{-x}-1+x.\]
Note that both these functions are increasing, and for $0\le x\le y$,
\begin{align}
& 0\le e_1^y-e_1^x\le y-x,\qquad 0\le e_2^x\le \min(x, \tfrac{1}{2}x^2),  \label{e12}\\
& e_1^{x+y}=e_1^x+e_1^y-e_1^xe_1^y,\qquad e_2^{x+y}=e_2^x+e_2^y+e_1^xe_1^y.\label{e1}
\end{align}

\item In different formulas, the symbols $C,C_1,C_2, c, c_1, c_2$ represent different positive constants.

\end{enumerate}

\section{Population counts}\label{psm}
The number of individuals  alive at time $t$ can be counted as the sum of individual scores
\[
Z(t)=\sum_{j=0}^t\sum_{k=1}^{Z_j}1_{\{j\le t<j+L_{jk}\}}=\sum_{j=0}^t\sum_{k=1}^{Z_j}\chi_{jk}(t-j),
\]
where $L_{jk}$  is the  life length of an individual born at time $j$, and $\chi_{jk}(t)=1_{\{0\le t<L_{jk}\}}$ is its individual score. In this case, the individual score is 1, if the individual is alive at time $t$, and 0 otherwise. This representation leads to the next definition of a population count.

\begin{definition}\label{def1}
For a progenitor of the GWO-process, define its individual score as a vector $(\chi(t))_{t\in\mathbb Z}$ with non-negative, possibly dependent components, such that $\chi(t)=0$ for all $t<0$. This random vector is allowed to depend on the individual characteristics   \eqref{ir}, but it is assumed to be independent from such characteristics of other individuals,

 Define a population count $X(t)= X^{[\chi]}(t)$ as the sum of  time shifted individual scores 
\begin{equation}
 X(t):=\sum_{j=0}^t\sum_{k=1}^{Z_j}\chi_{jk}(t-j),\quad t\in\mathbb Z,
\label{X}
\end{equation}
assuming that the individual scores $(\chi_{jk}(t))_{t\in\mathbb Z}$ are 
%functions of individuals' ages
%\[(\chi_{jk}(0),\chi_{jk}(1),\chi_{jk}(2),\ldots)\stackrel{\rm fdd}{=}(\chi(0),\chi(1),\chi(2),\ldots)\]
 independent copies of $(\chi(t))_{t\in\mathbb Z}$.
\end{definition}
%\subsection{Examples of population counts}\label{sec:ic}

\subsection{The litter sizes}\label{omg}

%%As before, denote by $Z_t$ the number of individuals born at time $t$ stemming from $Z_0$ ancestors born at time 0. In general, $Z_0, Z_1,\ldots$ is not a Markov chain. 
%

%Usually, it will be assumed that the reproduction process starts from a single individual born at time 0, so that $Z(0)=1$.

In terms of \eqref{ir}, the litter sizes of a generic individual are defined  by 
$\nu(t):=\sum_{j=1}^N1_{\{\tau_j=t\}}$, $t\ge1$, so that $\nu(1)+\ldots+\nu(L)=N$. On the other hand, given a random infinite dimensional vector 
\be\label{alt}
(L,\nu(1),\nu(2),\ldots),\quad L\ge1,\ \nu(t)\ge0,\ t\ge1,
\ee
where $\nu(t)$ is treated the litter size at age $t$ for an individual with the life length $L$, the consecutive ages at childbearing can be found as
$$\tau_j=\sum\nolimits_{t=1}^Lt1_{\{N(t-1)<\tau_j\le N(t)\}},\quad N(t):=(\nu(1)+\ldots+\nu(t))1_{\{L\ge t\}},$$
where $N(t)$ is the number of daughters produced by a mother of age $t$.

In the critical case, the probabilities
$$A(t):=\rE(\nu(t) 1_{\{L\ge t\}}),\quad t\ge1,$$
sum up to one, since
$\sum_{t\ge1} A(t)=\rE( \nu(1)+\ldots+\nu(L))=\rE(N)=1.$
A renewal process with inter-arrival times having distribution $A(1),(A(2),\ldots$ plays a crucial role in the analysis of the critical GWO-processes. Observe that the corresponding mean  inter-arrival time is indeed given by \eqref{aha}:
\[ \sum_{t=1}^\infty tA(t)=\rE\Big(\sum_{t=1}^\infty t\nu(t) 1_{\{L\ge t\}}\Big)=\rE\Big(\sum_{t=1}^\infty t\sum_{j=1}^N1_{\{\tau_j=t\}}\Big)=\rE\Big(\sum_{j=1}^N\sum_{t=1}^\infty t1_{\{\tau_j=t\}}\Big)=\rE(\tau_1+\ldots+\tau_N)=a.\]

\subsection{Associated renewal process}

In the GWO setting with $Z_0=1$, the process $Z_t$ conditioned on $\{N(t)=k\}$, can be viewed as the sum of $k$ independent daughter copies
$Z_t=Z^{(1)}_{t-\tau_j}+\ldots+Z^{(k)}_{t-\tau_{k}}$.
This branching property implies that the expected number of newborns  $U(t):=\rE_1(Z_t)$ satisfies  a recursive relation
\[U(t)=\rE\Big(\sum_{j=1}^{N(t)}U(t-\tau_j)\Big)=\rE\Big(\sum_{k=1}^tU(t-k)\nu(k) 1_{\{L\ge k\}}\Big)=U*A(t),\quad t\ge1,\]
where the $*$ symbol stands for a discrete convolution 
$$A_1*A_2(t):=\sum\nolimits_{j=-\infty}^\infty A_1(t-j)A_2(j),\quad t\in\mathbb Z.$$

Resolving the obtained recursion $U(t)=1_{\{t=0\}}+U*A(t),$
we find a familiar expression for the renewal function
\begin{equation}
  \label{Ut}
U(t)=1_{\{t=0\}}+\sum\nolimits_{k=1}^t A^{*k}(t),\quad  A^{*1}(t):=A(t),\ \ A^{*(k+1)}(t):=A^{*k}*A(t),
\end{equation}
so that by the elementary renewal theorem,
\begin{equation}
  \label{ert}
U(t)\to1/a,\quad t\to\infty.
\end{equation}
This says that in the long run, the underlying reproduction process produces one birth per $a$ units of time. In this sense, $a$ can be treated as the average generation length.

Later on, we will need the following facts concerning  the distribution of $W_t$, the waiting time to the next renewal event  %. The probability mass function of  $W_t$,
 \[
R_t(j):=\rP(W_t=j),\quad j\ge1,\ t\ge0.
\]
These probabilities satisfy  the renewal equation
$R_t(j)=A(t+j)+R_t*A(t),$
which yields 
\begin{equation}
  \label{rtj}
R_t(j)=\sum\nolimits_{k=0}^tA(t+j-k)U(k),\ j\ge1,\ t\ge0.
\end{equation}
By the key renewal theorem, there exists a stable distribution of the residual time $W_t$, in that
\begin{equation}\label{overshot}
 R_t(j)\to R(j),\quad t\to \infty,\quad   R(j):=a^{-1}\sum\nolimits_{k=j}^\infty A(k),\quad j\ge1.
\end{equation}

\begin{lemma}\label{rty} Assume \eqref{crit}, $a<\infty$, and suppose a family of non-negative functions $r^{(n)}(t)$ is such that 
\[\sup_{n\ge1,t\ge1}r^{(n)}(t)<\infty,\qquad r^{(n)}(ny)\stackrel{y\, }{\Rightarrow} r(y),\quad n\to\infty.\]
If $r(y)\to r(0)$ as $y\to0$, then
\[\sum\nolimits_{t=1}^\infty r^{(n)}(t)R_{ny}(t)\stackrel{y}{\to} r(0),\ n\to\infty.\]
\end{lemma}
\begin{proof} Observe that
\[\sum_{t=1}^{\infty} r^{(n)}(t)R_{ny}(t)-r(0)=\sum_{t=1}^{t_0} (r^{(n)}(t)-r(0))R_{ny}(t)+\sum_{t=t_0+1}^{\infty} (r^{(n)}(t)-r(0))R_{ny}(t)\]
for any  $t_0>0$. From
\[ \sum_{t=1}^{t_0} (r^{(n)}(t)-r(0))R_{ny}(t)= \sum_{t=1}^{t_0} (r^{(n)}(t)-r(tn^{-1}))R_{ny}(t)+\sum_{t=1}^{t_0} (r(tn^{-1})-r(0))R_{ny}(t), \]
we deduce
\[ \sum\nolimits_{t=1}^{t_0} (r^{(n)}(t)-r(0))R_{ny}(t)\stackrel{y}{\to} 0,\ n\to\infty, \]
using the assumptions on $r^{(n)}(\cdot)$ and $r(\cdot)$. It remains to notice that
\[\sum_{t=t_0+1}^{\infty} |r^{(n)}(t)-r(0)|R_{ny}(t)\le C\sum_{t=t_0+1}^{\infty} R_{ny}(t),\]
and
$\sum_{t=t_0+1}^{\infty} R_{ny}(t)\stackrel{y}{\to} \sum_{t=t_0+1}^{\infty} R(t)\to0$ as first $t\to\infty$ and then $t_0\to\infty$.
\end{proof}

\subsection{Expected population counts}

If $Z_0=1$, %meaning that the reproduction process starts with  a single individual born at time 0. Then 
then $X(t)$, defined by \eqref{X}, can be represented as 
\begin{equation}\label{basdec}
    X(t)=\chi(t)+\sum\nolimits_{j=1}^{N(t)}X^{(j)}(t-\tau_j)
\end{equation}
in terms of the independent daughter processes $X^{(j)}(\cdot)$.
Taking expectations, 
we arrive at a recursion
\[M(t)= m(t)+\rE\Big(\sum_{j=1}^{N(t)}M(t-\tau_j)\Big)= m(t)+\sum_{j=1}^{t}M(t-j)A(j),\]
where
$M(t):=\rE_1(X(t))$, $m(t):=\rE(\chi(t))$. 
This renewal equation
$M(t)=m(t)+M*A(t)$ yields 
$$M(t)=m*U(t)=\sum\nolimits_{j=0}^{t}m(t-j)U(j),$$
and applying the key renewal theorem, we conclude
 \begin{equation}\label{gM}
\rE_1(X(t))\to m_\chi,\quad t\to\infty,\quad m_\chi:=a^{-1}\sum\nolimits_{t=0}^\infty \rE(\chi(t)).
\end{equation}
The obtained parameter $m_\chi$ can be viewed as the average $\chi$-score for the population with overlapping generations. 
The next result  goes further than \eqref{gM} by giving a useful asymptotical relation in  the case $m_\chi=\infty$.

\begin{proposition}\label{rt}
Consider a critical GWO-process with $a<\infty$. If for some slowly varying at infinity function $\mathcal{L}(\cdot)$,
 \begin{equation}\label{gammo}
\sum\nolimits_{j=0}^t \rE(\chi(j))= t^\gamma\mathcal{L}(t),\ t\to\infty,\ 0\le\gamma<\infty,
\end{equation}
then 
 $\rE_1(X(t))\sim a^{-1}t^\gamma\mathcal{L}(t)$ as $t\to\infty$.
\end{proposition}

\begin{proof} 
We have to show that \eqref{gammo} implies
$M(t)- a^{-1} M_t=o(M_t)$ as $t\to\infty,$
where $M_t:=\sum_{j=0}^t m(j)$.
To this end, observe that the difference
\begin{eqnarray*}
 M(t)-a^{-1}\sum_{j=0}^t m(t-j)=\sum_{j=0}^t m(t-j)(U(j)-a^{-1})
 \end{eqnarray*} 
 is estimated from above by
\begin{align*}
 \sum_{j=0}^t m(t-j)\,|\,U(j)-a^{-1}\,|\,&\le C\sum_{j=0}^{t_\epsilon-1}m(t-j)+\epsilon\sum_{j=t_\epsilon}^{t}m(t-j)\le  C(M_t-M_{t-t_\epsilon})+\epsilon M_t,\quad t\ge t_\epsilon,
  \end{align*} 
for an arbitrarily small $\epsilon>0$ and some finite constants $C$, $t_\epsilon$. It remains to apply the property of the regularly varying function $M_t$, saying that $M_t-M_{t-c}=o(M_t)$ as $t\to\infty$ for any fixed $c\ge0$.

   \end{proof}

Turning to $X(t)=Z(t)$, the number of individuals alive at time $t$, observe that with $\chi(t)=1_{\{0\le t< L\}}$, 
$$\sum\nolimits_{t\ge0}\rE(\chi(t))=\sum\nolimits_{t\ge0}\rP(L>t)=\mu.$$ 
Therefore, given $\rE(N)=1$,
$$\rE_1(Z(t))\to \mu a^{-1},\quad t\to\infty.$$
In this case, the parameter $m_\chi=\mu a^{-1}$ can be treated as the degree of generation overlap. For example, $m_\chi=2$ means that on average, the life length $L$ covers two generation lengths.

\section{Branching renewal equations}\label{sec:bre}

A useful extension of Definition \ref{def1} broadens the range of individual scores by replacing \eqref{X} with 
 \begin{equation}
 X(t):=\sum_{j=0}^\infty  \sum_{k=1}^{Z_j}\chi_{jk}(t-j),\quad t\in\mathbb Z.
\label{X+}
 \end{equation}
Relation \eqref{X+} takes into account even those individuals who are born after time $t$, allowing $\chi(t)>0$ for $t<0$. %Definition \ref{def1} based on relation \eqref{X} represents the special case when $X(t)\equiv0$ for $t<0$.
In this paper, we refer to this extension only to deal with the finite dimensional distributions of the population counts defined by \eqref{X}, see Lemma \ref{lel} below. 
%To see this, observe that given \eqref{X}, the time shifted count $X(s+t)$ can be treated as $X^{[\psi]}(t)$ with $\psi(t)=\chi(s+t)$ for $t\ge-s$.
\begin{definition}  \label{lala}
 For the population count $X(t)= X^{[\chi]}(t)$ given by \eqref{X+}, define a log-Laplace transform $\Lambda(t)=\Lambda^{[\chi]}(t)$ via
\begin{equation*}
e^{-\Lambda(t)}:=\rE_1(e^{-X(t)}),\quad t\in\mathbb Z.
\end{equation*}

\end{definition}
The purpose of this section is to introduce a branching renewal equation for $\Lambda(\cdot)$ and establish Proposition \ref{cth}, which will play a key role in the proofs of the main results of this paper.

\begin{lemma}\label{lel}
For a given vector  $(t_1,\ldots,t_p)$ with non-negative integer components, consider the log-Laplace transform
 \begin{equation*}
\Lambda(t)=-\ln\rE_1\Big(\exp\Big\{-\sum_{i=1}^p\lambda_iX(t_i+ t)\Big\}\Big)
\end{equation*}
 of the $p$-dimensional distribution of the population sum $X(\cdot)$ defined by \eqref{X}. 
Then, in accordance with Definition \ref{lala},
 $$\Lambda(t)=\Lambda^{[\psi]}(t),\qquad \psi(t):=\sum_{i=1}^p\lambda_i\chi(t_i+t),\quad t\in\mathbb Z.$$
\end{lemma}

\begin{proof}
It suffices to observe that  
 \[\sum_{i=1}^p\lambda_iX(t_i+ t)\stackrel{\eqref{X}}{=}\sum_{i=1}^p\sum_{j=0}^t  \sum_{k=1}^{Z_j}\lambda_i\chi_{jk}(t_i+t-j)=\sum_{j=0}^\infty  \sum_{k=1}^{Z_j}\psi_{jk}(t-j)\stackrel{\eqref{X+}}{=}X^{[\psi]}(t).
\]
%Part (b) is obtained from the split $\psi_p(t)=\lambda_p\chi(t_p+t)+\sum_{i=1}^{p-1}\lambda_i\chi(t_i+t)$ using \eqref{e1}.
\end{proof}

\subsection{Derivation of the branching renewal equation}

Here we show that Definition \ref{lala} leads to what we call a branching renewal equation:
\begin{equation}
  \label{NRE}
 \Lambda(t)=B(t)-\Psi[\Lambda]*U(t),\quad t\ge0,
\end{equation}
where the operator
\begin{equation}
  \label{Psi}
 \Psi[f](t):=\rE\Big(\prod_{j=1}^L e^{-\nu(j)f(t-j)}\Big)-\sum_{j=1}^\infty e^{-f(t-j)}A(j),\quad t\ge0
\end{equation}
is defined on the set of  non-negative sequences $(f(t))_{t\in\mathbb Z}$, see more on it in Section \ref{sec:psi}. The convolution term $\Psi[\Lambda]*U(t)$ represents the non-linear part of the branching renewal equation. A seemingly free term $B(\cdot)$ of the equation \eqref{NRE} is a non-negative function specified below by \eqref{dy} and \eqref{by}. It also depends on the function $\Lambda(\cdot)$ in a non-linear way, however, asymptotically it acts as a truly free term.

The derivation of  \eqref{NRE} is based on the following extended version of decomposition \eqref{basdec} 
\[
    X(t)=\chi(t)+\sum\nolimits_{j=1}^{N}X^{(j)}(t-\tau_j),\quad t\in\mathbb Z,
\]
where $X^{(j)}(\cdot)$ are independent daughter copies of $(X(\cdot)|Z_0=1)$.
It entails
 $e^{\chi(t)-X(t)}=\prod_{j=1}^{N}e^{-X^{(j)}(t-\tau_j)}$,
and taking expectations, we obtain
\[    \rE_1(e^{\chi(t)-X(t)})=\rE(e^{-\sum_{j=1}^{N}\Lambda(t-\tau_j)})=\rE(e^{-\sum_{j=1}^L \nu(j)\Lambda(t-j)}).\]
On the other hand (recall $e_1^x:=1-e^{-x}$),
\begin{align*}
 \rE_1(e^{\chi(t)-X(t)})- e^{-\Lambda(t)}&=\rE_1(e^{\chi(t)-X(t)}-e^{-X(t)})=\rE_1\Big(e_1^{\chi(t)}e^{\chi(t)-X(t)}\Big).
\end{align*}
Denoting the last expectation $D(t)$, we can write
\begin{equation}
  \label{dy}
D(t)=\rE\Big(e_1^{\chi(t)}e^{-\sum_{j=1}^L\nu(j)\Lambda(t-j)}\Big),
\end{equation}
due to independence between the progenitor score $\chi(t)$ and the GWO-processes stemming from progenitor's daughters.
Combing the previous relations, we find
\[e^{-\Lambda(t)}=\rE(e^{-\sum_{j=1}^L \nu(j)\Lambda(t-j)})-D(t),\]
which after introducing a term involving operator \eqref{Psi}, brings
\[   e^{-\Lambda(t)}=\sum\nolimits_{j=1}^\infty e^{-\Lambda(t-j)}A(j)+\Psi[\Lambda](t)-D(t).\]

Subtracting both sides from 1, yields 
\[   e_1^{-\Lambda(t)}=\sum\nolimits_{j=1}^\infty e_1^{-\Lambda(t-j)}A(j)-\Psi[\Lambda](t)+D(t),\]
which can be rewritten in the form of a renewal equation
\[   e_1^{-\Lambda(t)}=e_1^{-\Lambda}*A(t)+\sum_{j=t+1}^\infty e_1^{-\Lambda(t-j)}A(j)-\Psi[\Lambda](t)+D(t).\]
Formally solving this renewal function, we get
 \begin{equation}
  \label{bre1}
 e_1^{-\Lambda(t)}=\sum_{j=1}^\infty e_1^{\Lambda(-j)}R_t(j)-\Psi[\Lambda]*U(t)+D*U(t),
\end{equation}
 where $R_t(j)$ is given by \eqref{rtj}. Here we used
\[  \sum_{k=0}^t \sum_{j=t-k+1}^\infty e_1^{-\Lambda(t-k-j)}A(j)U(k)= \sum_{k=0}^t U(k)\sum_{j=1}^\infty e_1^{-\Lambda(-j)}A(j+t-k)=\sum_{j=1}^\infty e_1^{\Lambda(-j)}R_t(j).\]
Since $e_1^{-\Lambda(t)}=\Lambda(t)-e_2^{-\Lambda(t)}$, we conclude that relation \eqref{NRE} holds with
 \begin{equation}
  \label{by}
 B(t)=e_2^{\Lambda(t)}+\sum_{j=1}^\infty e_1^{\Lambda(-j)}R_t(j)+D*U(t).
\end{equation}

\subsection{Laplace transform of the reproduction law}\label{sec:psi}
The Laplace transform of the reproduction law $\rE\left(e^{-f(\tau_1)-\ldots-f(\tau_N)}\right)$ is a positive functional defined on the set of non-negative sequences $(f(t))_{t\ge1}$. The higher than first moments of the joint distribution of $(\tau_1,\ldots,\tau_N)$ are characterised by the non-linear functional
\begin{equation}
  \label{Psif}
 \Psi(f):=\rE\Big(\prod_{j=1}^N e^{-f(\tau_j)}-\sum_{j=1}^N e^{-f(\tau_j)}\Big).
\end{equation}
This functional  is non-negative and monotone in view of the  elementary equality
\begin{equation}
  \label{elin}
\sum_{i=1}^k(a_i-b_j)-\prod_{i=1}^ka_i+\prod_{i=1}^kb_i=\sum_{i=1}^k(a_i-b_j)(1-a_1\ldots a_{j-1}b_{j+1}\ldots b_k).
\end{equation}
Earlier introduced operator  \eqref{Psi} is obtained from functional (\ref{Psif}) through the connection
$$\Psi[f](t)=\Psi(f_t),\quad f_t(j):=f(t-j)1_{\{1\le j\le t\}},$$
which is verified by
\begin{align*}
  \Psi(f_t)&\stackrel{(\ref{Psif})}{=}\rE\Big(\prod_{j=1}^N e^{-f_t(\tau_j)}-\sum_{j=1}^N e^{-f_t(\tau_j)}\Big)=\rE\Big(\prod_{k=1}^L e^{-f_t(k)\nu(k)}-\sum_{k=1}^L e^{-f_t(k)}\nu(k)\Big)\\
  &=\rE\Big(\prod_{k=1}^L e^{-f(t-k)\nu(k)}\Big)-\sum_{k=1}^\infty e^{-f(t-k)}A(k)\stackrel{(\ref{Psi})}{=}\Psi[f](t).
\end{align*}

 \begin{lemma}\label{nelli} Consider a constant function $f(t)=z$,  $t\in\mathbb Z$.  If \eqref{crit}, then
 \[ \Psi[f](t)=\Psi(z)=\rE(e^{-z N})-e^{-z},\quad t\ge0,\]
and $z^{-2}\Psi(z)\to b$ as $z\to0$.
 \end{lemma}
 
\begin{proof}
The first  assertion follows from the relation connecting $\Psi[f](t)$ and $\Psi(f)$. The second assertion follows from the L'Hospital rule.
\end{proof}

 \begin{lemma}\label{Lpsi}   If (\ref{crit}) holds, and 
 $$nr_n(ny)\stackrel{y\,}{\Rightarrow}r(y),\quad n\to\infty,$$
 where $r:[0,\infty)\to[0,\infty)$ is a continuous function, then 
 \[n^2\Psi[r_n](ny)\stackrel{y}{\to} br^2(y),\ n\to\infty.\]
 \end{lemma}
\begin{proof}  Observe, that (\ref{elin}) implies 
 \begin{equation*}
%   \label{camilla}
\Psi[f](t)-\Psi[g](t)=\rE\Big(\sum_{j=1}^{N} \Big(e^{-g(t-\tau_j)}-e^{-f(t-\tau_j)}\Big)\Big(1-\prod_{i=1}^{j-1} e^{-f(t-\tau_i)}\prod_{i=j+1}^{N}  e^{-g(t-\tau_i)}\Big)\Big),
\end{equation*}
which in turn gives for arbitrary $1\le t_1\le t$,
\begin{align*}
\,|\,\Psi[f](t)-\Psi[g](t)\,|\, &\le \rE\Big(\sum_{j=1}^{N(t_1)} |f(t-\tau_j)-g(t-\tau_j)| I_j+\|f\vee g\|\sum_{j=N(t_1)+1}^{N} I_j\Big),
\end{align*}
where $\|f\|:=\sup_{t\ge1}\,|\,f(t)\,|$ and
\[I_j:= \Big(1-\prod_{i=1}^{j-1} e^{-f(t-\tau_i)}\prod_{i=j+1}^{N}  e^{-g(t-\tau_i)}\Big)\le \sum_{i=1}^{j-1} f(t-\tau_i)+\sum_{i=j+1}^{N}  g(t-\tau_i)\le \|f\vee g\|(N-1).\]
Using $\rE(N(N-1))=2b$, we therefore obtain
\begin{align*}
\,|\,\Psi[f](t)-\Psi[g](t)\,|\,&\le  2b\|f\vee g\|\,\max_{1\le j\le t_1}\,|\,f(t-j)-g(t-j)\,|\, +\|f\vee g\|^2 \rE((N(t)-N(t_1))N).
\end{align*}
This implies that 
\begin{equation}
  \label{dpsi}
 \,|\,\Psi[f](t)-\Psi[g](t)\,|\,\le2b\|f\vee g\|\max_{1\le j\le t_1}\,|\,f(t-j)-g(t-j)\,|\,+\|f\vee g\|^2\delta(t_1),
\end{equation}
where 
$\delta(t):=\rE((N-N(t))N)\to0$ as $t\to\infty.$
% In particular, with $t=t_1$ we get
%\begin{equation}
%  \label{dpsi=}
% \,|\,\Psi(f)(t)-\Psi(g)(t)\,|\,\le2b\|f\vee g\|\max_{1\le j\le t}\,|\,f(j)-g(j)\,|\,+\|f\vee g\|^2\delta(t).
%\end{equation}

Applying (\ref{dpsi}) with $ t_1=n\epsilon$, $t=ny$, and
$$f(j):=r_n(j),\quad g(j):=z_n,\quad j\ge 1,\quad z_n:=n^{-1}r(y),$$ 
%that is assuming the function $g(\cdot)$ being a constant $z_n$, 
we get
 \[\,|\, n^2\Psi[r_n](ny)-n^2\Psi[z_n](ny)\,|\,\le C\sup_{0\le x\le \epsilon}\,|\,nr_n(n(y-x))-r(y)\,|\,+C_1\delta(n\epsilon).\]
 Thus, under the  imposed conditions, 
 \[\lim_{\epsilon\to0}\sup_{0\le y\le y_0}(n^2\Psi[r_n](ny)-n^2\Psi[z_n](ny))\to0, \quad n\to\infty,\]
 for any  $y_0>0$. 
 It remains to observe that 
$n^2\Psi[z_n](ny)\stackrel{y}{\to} br^2(y)$ as $n\to\infty$, according to Lemma \ref{nelli}.
\end{proof}

\subsection{Basic convergence result}

If $\Lambda(t)$ is given by Definition \ref{lala}, then
\begin{equation}
  \label{layla}
\rE_n(e^{-X(t)})=e^{-n\Lambda(t)},
\end{equation}
This observation explains the importance of the next result.

\begin{proposition} \label{cth}
Assume (\ref{crit}), $a<\infty$, and consider a sequence of positive functions $\Lambda_n(\cdot)$ satisfying  
\begin{equation}
\label{sNRE}
\Lambda_n(t)=B_n(t)-\Psi(\Lambda_n)*U(t),\quad t\ge0,\quad n\ge1.
\end{equation}
If the non-negative functions $B_n(t)$ are such that
\begin{equation}\label{btn}
    nB_n(ny)\stackrel{y}{\to}B(y),\quad n\to\infty,
\end{equation}
where $B(y)$ is a  continuous function, then 
$$n\Lambda_n(ny)\stackrel{y}{\to}r(y),\quad n\to\infty,$$
where $r(y)$ is a continuous function uniquely defined by 
\be\label{rc}
r(y)=B(y)-ba^{-1}\int_0^yr^2(u)du.
\ee
\end{proposition}
\begin{proof}. We will prove this statement  in three steps. Firstly, we will show
\be\label{bro}
r(y)=nB_n(ny)-n\sum_{t=0}^{ny}\Psi[n^{-1}r_n](ny-t)U(t)+\delta_n(y),
\ee
where $\delta_n(y)$ stands for a function (different in different formulas) such that $\delta_n(y)\stackrel{y}{\to}0$ as $n\to\infty$.  
Secondly, putting
$\Delta_n(y):=n\Lambda_n(ny)-r(y),$
we will find a $y^*>0$ such that
\be\label{nasos}
\sup_{y_0\le u\le y_1}\,|\,\Delta_n(u)\,|\,\to0,\ n\to\infty,\ 0<y_0\le y_1\le y^*.
\ee
Thirdly, we will demonstrate that
\be\label{naos}
\Delta_n(y)\stackrel{y}{\to}0,\quad n\to\infty.
\ee

Proof of (\ref{bro}).  Rewriting (\ref{rc}) as
$r(y)=B(y)-b\int_0^yr^2(y-u)a^{-1}du$,
and using (\ref{ert}),  (\ref{btn}), we obtain 
\[r(y)=nB_n(ny)-bn^{-1}\sum_{t=0}^{ny}r^2(y-tn^{-1})U(t)+\delta_n(y).\]
This and Lemma \ref{Lpsi} imply (\ref{bro}). 

Proof of (\ref{nasos}). Relations  (\ref{sNRE})  and (\ref{bro}) yield
\begin{equation}\label{Del}
    \Delta_n(y)=n\sum_{t=0}^{ny}\Big(\Psi[\Lambda_n](t)-\Psi[n^{-1}r_n](t)\Big)U(ny-t)+\delta_n(y).
\end{equation}
Under the current assumptions, the inequality $n\Lambda_n(ny)\le nB_n(ny)$ implies that the sequence of functions $n\Lambda_n(ny)$ is uniformly bounded over any  finite interval $0\le y\le y_1$. Therefore, putting  $t_1:=t\epsilon$ into (\ref{dpsi}) gives
\[n^2\,|\,\Psi[\Lambda_n](t)-\Psi[n^{-1}r_n](t)\,|\,\le C_1\sup_{(1-\epsilon)t\le j\le t}\,|\,\Delta_n(jn^{-1})\,|\,+C_2\delta(t\epsilon),\]
for any fixed $0<\epsilon<1$.
Combining this with (\ref{Del}), entails
\begin{align}\label{victory}
 \,|\,\Delta_n(y)\,|\,&\le Cn^{-1}\sum_{t=n\epsilon}^{ny}U(ny-t)\sup_{(1-\epsilon)t\le j\le t}\,|\,\Delta_n(jn^{-1})\,|\,+C_1n^{-1}\sum_{t=0}^{n\epsilon}U(ny-t)+\delta_n(y),
 \end{align}
 so that for some positive constant $c^*$ independent of $(n,\epsilon,y)$,
\begin{align}\label{deL}
 \,|\,\Delta_n(y)\,|\,\le c^* y\sup_{\epsilon(1-\epsilon)\le u\le y}\,|\,\Delta_n(u)\,|\,+C\epsilon+\delta_n(y).
\end{align}
It follows, 
\[\sup_{\epsilon(1-\epsilon)\le y\le v}\,|\,\Delta_n(y)\,|\,\le c^*v\sup_{\epsilon(1-\epsilon)\le u\le v}\,|\,\Delta_n(u)\,|\,+C\epsilon+\sup_{\epsilon(1-\epsilon)\le y\le v}\delta_n(y). \]
%so that
%\[\sup_{\epsilon(1-\epsilon)\le u\le y}\,|\,\Delta_n(u)\,|\,\le c^*y\sup_{\epsilon(1-\epsilon)\le u\le y}\,|\,\Delta_n(u)\,|\,+C\epsilon+\sup_{\epsilon(1-\epsilon)\le u\le y}\delta_n(u). \]
Replacing here  $v$ by $y^*:=(2c^*)^{-1}$, we derive
\[\limsup_{n\to\infty}\sup_{\epsilon(1-\epsilon)\le u\le y^*}\,|\,\Delta_n(u)\,|\,\le C\epsilon,\]
which, after letting $\epsilon\to0$, results in (\ref{nasos}).

Proof of (\ref{naos}). It suffices to demonstrate that the convergence interval in  (\ref{nasos}) can be consecutively expanded from $(0, y^*]$ to $(0,2y^*]$, from $(0,2y^*]$ to $(0,3y^*]$, and so forth. Suppose we have established, that for some $k\ge1$,
\begin{align*}\label{veliky}
\sup_{y_0\le u\le y_1}\,|\,\Delta_n(u)\,|\,\to0,\ n\to\infty,\ 0<y_0\le y_1\le ky^*.
\end{align*}
Then  for $ky^*<y\le(k+1)y^*$, by (\ref{victory}),
\begin{align*}
 \,|\,\Delta_n(y)\,|\,&\le Cn^{-1}\sum_{t=nky^*}^{ny}U(ny-t)\sup_{(1-\epsilon)t\le j\le t}\,|\,\Delta_n(jn^{-1})\,|\,+C\epsilon+\delta_n(y),
 \end{align*}
yielding
\begin{eqnarray*}%
 \sup_{ky^*\le y\le (k+1)y^*}\,|\,\Delta_n(y)\,|\,%&\le& c^*(y-ky^*)^\alpha\sup_{\epsilon(1-\epsilon)\le u\le y}\,|\,\Delta_n(u)\,|\,+C_1\epsilon+\delta_n(y)\\
 \le c^*y^*\sup_{ky^*\le u\le (k+1)y^*}\,|\,\Delta_n(u)\,|\,+C\epsilon+\sup_{ky^*\le u\le (k+1)y^*}\delta_n(y).
\end{eqnarray*}
Since $c^*y^*<1$, we may conclude that 
$$\sup_{ky^*\le u\le (k+1)y^*}\,|\,\Delta_n(u)\,|\,\to0,\ n\to\infty,$$
thereby completing the proof of  (\ref{naos}).
\end{proof}

\section{Continuous state critical branching process }\label{secta}
In this section, among other things, we clarify the meaning of $\xi_\gamma(\cdot)$ given by \eqref{asmaa}, in terms of the log-Laplace transforms of the fdd's of the process $\xi(\cdot)$. From now on we  consistently use the following shortenings
\begin{align*}
 G_p(\bar u,\bar\lambda)&:=G_p(u_1,\ldots,u_p;\lambda_1,\ldots,\lambda_p),\\
 G_p(c_1\bar u+y,c_2\bar\lambda)&:=G_p(c_1u_1+y,\ldots,c_1u_p+y;c_2\lambda_1,\ldots,c_2\lambda_p),\\
 H_{p,q}(\bar u,\bar\lambda)&:=H_{p,q}(u_1,\ldots,u_p;\lambda_{11},\ldots,\lambda_{p1};\ldots;\lambda_{1q},\ldots,\lambda_{pq}).
\end{align*}

\subsection{Laplace transforms for $\xi(\cdot)$}\label{sec:cb}
The set of functions 
 \begin{equation} \label{gula}
G_p(\bar u,\bar\lambda):=- \ln\rE_1\Big(e^{-\lambda_1\xi(u_1)-\ldots-\lambda_p\xi(u_p)}\Big),\quad p\ge1,
\end{equation}
with $u_i,\lambda_i\ge0$, determines the fdd's for the process $\xi(\cdot)$.

\begin{lemma}\label{lla}
For non-negative $x,y, u_1, u_2,\ldots,\lambda_1,\lambda_2,\ldots$,
\[\rE\Big(e^{-\lambda_1\xi(u_1+y)-\ldots-\lambda_{p}\xi(u_{p}+y)}\,|\,\xi(y)=x\Big)=e^{-xG_{p}(\bar u,\bar\lambda)}.\]
\end{lemma}
\begin{proof} This result is obtained by induction, using \eqref{cbp} and the Markov property of $\xi(\cdot)$. To illustrate the argument, take $p=2$ and non-negative $y, y_1, y_2$. We have
\begin{align*}
 \rE\Big(e^{-\lambda_1\xi(y+y_1+y_2)-\lambda_{2}\xi(y+y_{2})}\,|\,\xi(y)=x\Big)
& =\rE\Big(e^{-\lambda_{2}\xi(y+y_{2})}\rE\Big(e^{-\lambda_1\xi(y+y_1+y_2)}\,|\,\xi(y+y_2)\Big)\,|\,\xi(y)=x\Big)\\
%& =\rE\Big(e^{-\lambda_{2}\xi(u_{2})}\exp\Big\{-\frac{\lambda_1\xi(u_{2})}{1+b\lambda_1(u_1-u_2)}\Big\}\,|\,\xi(y)=x\Big)\\
& \stackrel{(\ref{cbp})}{=}\rE\Big(\exp\Big\{-(\lambda_{2}+\tfrac{\lambda_1}{1+b\lambda_1 y_1})\xi(y+y_{2})\Big\}\,|\,\xi(y)=x\Big)\\
&\stackrel{(\ref{cbp})}{=}\exp\Big\{-\frac{(\lambda_1+\lambda_{2}+b\lambda_1\lambda_{2}y_1)x}{1+b\lambda_1(y_1+y_2)+b\lambda_{2}y_2+b^2\lambda_1\lambda_2y_1y_2}\Big\}.
\end{align*}
With $u_2=y_2$ and $u_1=y_1+y_2$, this gives an explicit expression
\[G_2(\bar u,\bar\lambda)=\frac{\lambda_1+\lambda_{2}+b\lambda_1\lambda_{2}(u_1-u_2)}{1+b\lambda_1u_1+b\lambda_{2}u_2+b^2\lambda_1\lambda_2(u_1-u_2)u_2},\]
for the asserted relation
$\rE\Big(e^{-\lambda_1\xi(u_1+y)-\lambda_{2}\xi(u_{2}+y)}\,|\,\xi(y)=x\Big)=e^{-xG_{2}(\bar u,\bar\lambda)}$ in the case $p=2$.
 \end{proof}

\begin{lemma}\label{illa}
If
\begin{equation} \label{ula}
u_1>\ldots>u_p=0,\quad \lambda_1\ge0,\ldots,\lambda_p\ge0,
\end{equation}
then for all $y\ge0,$ assuming $G_0(\bar u,\bar\lambda):=0$, the following two relations hold
\begin{align}
 G_p(\bar u+y,\bar\lambda) &= (by+(G_{p-1}(\bar u,\bar\lambda)+\lambda_p)^{-1})^{-1},%= \frac{G_{p-1}(\bar u,\bar\lambda)+\lambda_p}{1+by(G_{p-1}(\bar u,\bar\lambda)+\lambda_p)}
\nonumber \\
  G_p(\bar u+y,\bar\lambda)&= G_{p-1}(\bar u,\bar\lambda)+\lambda_p-b\int_0^yG_p^2(\bar u+v,\bar\lambda)dv. \label{Gp}
\end{align}
\end{lemma}
 
\begin{proof}
 With $u_p=0$, relation \eqref{gula} gives
 \begin{align*}
 G_p(\bar u+y,\bar\lambda) &= -\ln \rE_1\Big(e^{-\lambda_p\xi(y)}\rE\Big(e^{-\lambda_1\xi(u_1+y)-\ldots-\lambda_{p-1}\xi(u_{p-1}+y)}\,|\,\xi(y)\Big)\Big). 
\end{align*}
Applying Lemma \ref{lla} and \eqref{cbp}, we get the first statement
\begin{align*}
 G_p(\bar u+y,\bar\lambda) &=-\ln \rE_1\Big(e^{-\lambda_p\xi(y)}e^{-G_{p-1}(\bar u,\bar\lambda)\xi(y)}\Big)=
(by+(G_{p-1}(\bar u,\bar\lambda)+\lambda_p)^{-1})^{-1}.
\end{align*}
To arrive at the second statement, it is enough to verify that the function $H(y)= (by+H_0^{-1})^{-1}$ satisfies
\begin{equation}
  \label{ric}
H(y)=H_0-b\int_0^yH^2(v)dv.
\end{equation}
\end{proof}

\subsection{Riccati integral equations}\label{sec:cb}

Equation \eqref{Gp} has a form of the Riccati integral equation \eqref{ric}, associated with a simple Riccati differential equation
$H'(y)=-H^2(y)$,  $ H(0)=H_0$.
Our limit theorems require a  more general equation of this type
\be\label{heq}
H(y)=F(y)-b\int_0^yH^2(v)dv.
\ee

\begin{lemma}\label{love}
Let function $F:[0,\infty)\to[0,\infty)$ be non-decreasing, with $F(0)\ge0$. 
For a given $n\ge1$, consider the step function
$$F^{(n)}(y):=\sum_{k=0}^\infty F(\tfrac{k}{n})1_{\{{k\over n}\le y<{k+1\over n}\}},\quad y\ge0,$$
and put 
\[e^{-H^{(n)}(y)}:=\rE_1 \Big(\exp\{-\xi\circ F^{(n)}(\tfrac{ny}{n})\}\Big), \quad y\ge0,\]
where 
\[\xi\circ F^{(n)}(\tfrac{k}{n}):=\xi(\tfrac{k}{n})F(0)+\sum_{i=1}^{k} \xi(\tfrac{k-i}{n})(F(\tfrac{i}{n})-F(\tfrac{i-1}{n})).\]
Then the function $H^{(n)}(\cdot)$ satisfies a recursion
\begin{align*}
  H^{(n)}(\tfrac{k}{n})=F(\tfrac{k}{n})-F(\tfrac{k-1}{n})+H^{(n)}(\tfrac{k-1}{n})(1+\tfrac{b}{n}H^{(n)}(\tfrac{k-1}{n}))^{-1},\quad k\ge1,
  \end{align*}
  with $H^{(n)}(0)=F(0)$.
\end{lemma}

\begin{proof}
Putting $f_k:=F(\tfrac{k}{n})$ and $f_{-1}:=0$, we get  
\begin{align*}
 H^{(n)}(\tfrac{k}{n})%&=-\ln\rE_1\Big(e^{-\int_0^{\tfrac{k}{n}}\xi(\tfrac{k}{n}-v)dF^{(n)}(v)}\Big) 
 &=-\ln\rE_1\Big(\exp\Big\{-\sum_{i=0}^k \xi(\tfrac{k-i}{n})(f_i-f_{i-1})\Big\}\Big)\\
 &=f_k-f_{k-1}-\ln\rE_1\Big(\exp\Big\{-\sum_{i=0}^{k-1} \xi(\tfrac{k-i}{n})(f_i-f_{i-1})\Big\}\Big),
  \end{align*}
  and by Lemma \ref{lla},
\begin{align*}
 H^{(n)}(\tfrac{k}{n})=f_k-f_{k-1}+G_{k}(\bar u+\tfrac{1}{n},\bar\lambda),
  \end{align*}
  with
$u_i:= \tfrac{k-i}{n}$ and $\lambda_i=f_{i-1}-f_{i-2}$ for $i\ge1$.
Since by Lemma  \ref{illa},
\[G_{k}(\bar u+\tfrac{1}{n},\bar\lambda)= (\tfrac{b}{n}+(G_{k-1}(\bar u,\bar\lambda)+\lambda_k)^{-1})^{-1},\]
we conclude
\begin{align*}
  H^{(n)}(\tfrac{k}{n})=f_k-f_{k-1}+(\tfrac{b}{n}+H^{-1}_n(\tfrac{k-1}{n}))^{-1}=f_k-f_{k-1}+H^{(n)}(\tfrac{k-1}{n})(1+\tfrac{b}{n}H^{(n)}(\tfrac{k-1}{n}))^{-1}.
  \end{align*}

\end{proof}
\begin{proposition}\label{inteq}
Let function $F(\cdot)$ have a continuous derivative $F':[0,\infty)\to[0,\infty)$ and let $F(0)\ge0$. The functions $H^{(n)}(\cdot)$, defined by Lemma \ref{love}, converge
\[H^{(n)}(y)\to H(y),\quad y\ge0,\quad n\to\infty,\]
to the solution of the Riccati  equation \eqref{heq}.
\end{proposition}
\begin{proof}
Applying a Taylor expansion to the recursion of Lemma \ref{love}, we obtain
\begin{align*}
H^{(n)}(\tfrac{k}{n})&=f_k-f_{k-1}+H^{(n)}(\tfrac{k-1}{n})-\tfrac{b}{n}(H^{(n)}(\tfrac{k-1}{n}))^2+\epsilon_n(k),\\
\epsilon_n(k)&=H^{(n)}(\tfrac{k-1}{n})\Big((1+\tfrac{b}{n}H^{(n)}(\tfrac{k-1}{n}))^{-1}-1+\tfrac{b}{n}H^{(n)}(\tfrac{k-1}{n})\Big)
=\frac{(\tfrac{b}{n})^2(H^{(n)}(\tfrac{k-1}{n}))^3}{1+\tfrac{b}{n}H^{(n)}(\tfrac{k-1}{n})}.
 \end{align*}
By reiterating this recursion, we get
\be\label{Heq}
H^{(n)}(\tfrac{k}{n})=f_k-\tfrac{b}{n}\sum_{i=0}^{k-1}(H^{(n)}(\tfrac{i}{n}))^2+\sum_{i=1}^{k}\epsilon_n(i).
\ee
To prove the lemma, it suffices to verify that
\be\label{sist}
 \Delta_n(k):=H^{(n)}(\tfrac{k}{n})-H(\tfrac{k}{n})\stackrel{y\,}{\Rightarrow}0,\quad n\to\infty,
 \ee
where $H^{(n)}(\cdot)$ satisfies \eqref{Heq}, with $f_i=F(\tfrac{i}{n})$. 
To this end, note that
\[\sum_{i=0}^{k} \xi(\tfrac{k-i}{n})(f_i-f_{i-1})=f_{k}\xi(0)+\sum_{i=0}^{k-1} (\xi(\tfrac{k-i}{n})-\xi(\tfrac{k-i-1}{n}))f_i\le f_{k}\xi(\tfrac{k}{n})\]
implies an upper bound
\[H^{(n)}(\tfrac{k}{n})\le -\ln \rE_1\Big(e^{-f_{k}\xi(\tfrac{k}{n})}\Big)\stackrel{(\ref{cbp})}{=} \frac{f_k}{1+bf_k\tfrac{k}{n}}\,,
 \]
that ensures $H^{(n)}(\tfrac{k}{n})\le C(y)$, provided $f_k\le C_1(y)$ for all $k\le ny$, so that
$\sum_{i=1}^{ny}\epsilon_n(i)\stackrel{y\,}{\Rightarrow}0$ as $n\to\infty$.

This and \eqref{Heq} entail
\[\Delta_n(\tfrac{k}{n})=-\tfrac{b}{n}\sum_{i=0}^{k-1}\Delta_n(\tfrac{i}{n})(H^{(n)}(\tfrac{i}{n})+H(\tfrac{i}{n}))+\delta_n(k),\]
where $\delta_n(ny)\stackrel{y\,}{\Rightarrow}0$ as $n\to\infty$. In view of this relation, we can find a sufficiently small $y^*>0$, such that 
\[\sup_{0\le y\le y^*} |\Delta_n(ny)|\to0,\quad n\to\infty.\]
It follows,
\[\Delta_n(\tfrac{k}{n})=-\tfrac{b}{n}\sum_{i=ny^*}^{k-1}\Delta_n(\tfrac{i}{n})(H^{(n)}(\tfrac{i}{n})+H(\tfrac{i}{n}))+\delta'_n(k),\]
where $\delta'_n(ny)\stackrel{y\,}{\Rightarrow}0$ as $n\to\infty$. This in turn, gives 
\[\sup_{0\le y\le 2y^*} |\Delta_n(ny)|\to0,\quad n\to\infty,\]
and proceding in the same manner, we arrive at \eqref{sist}.

\end{proof}

\subsection{Laplace transforms for $\xi\circ F(\cdot)$}\label{sec:cb}
Notice that the Riemann-Stieltjes integrals appearing in this paper are understood as
\[
\int_0^tf(u)dF(u):=F(0)f(0)+\int_{(0,t]}f(u)dF(u).
\]
Referring to Proposition \ref{inteq}, we treat the Riemann-Stieltjes integral 
$$\xi\circ F(y)=\int_0^{y}\xi(y-v)dF(v)$$
 as a random variable satisfying
$\rE_1(e^{-\xi\circ F(y)})=e^{-H(y)}$. 
This interpretation will be extended to the fdd's of $\xi\circ F(\cdot)$ in terms of the log-Laplace transforms 
\begin{equation}
  \label{dHp}
H_p(\bar u,\bar\lambda):=-\ln\rE_1\Big(e^{-\lambda_1\xi\circ F(u_1)-\ldots-\lambda_p\xi\circ F(u_p)}\Big).
\end{equation}

\begin{lemma} \label{lilla}
Under the assumptions of Proposition \ref{inteq}, given \eqref{ula}, the function  \eqref{dHp} satisfies
\begin{equation}
  \label{Hp}
H_p(\bar u+y,\bar\lambda)= H_{p-1}(\bar u,\bar\lambda)+F^\circ_p(y)-b\int_0^yH_p^2(\bar u+v,\bar\lambda)dv,
\end{equation}
where 
$F^\circ_p(y):=\sum_{i=1}^{p}\lambda_i (F(u_i+y)-F(u_i))$ for $y>0$, and $F^\circ_p(0):=\lambda_pF(0)$.
\end{lemma}
\begin{proof} 
The proof of Lemma  \ref{lilla} uses similar argument as Lemma  \ref{love} and Proposition \ref{inteq}, with the main idea being
to demonstrate that the step function version of \eqref{dHp}, defined by
\[e^{-H_{p}^{(n)}(\bar u+y,\bar\lambda)}:=\rE_1\Big(\exp\Big\{-\sum_{j=1}^p\lambda_i\xi\circ F^{(n)}(\tfrac{nu_i}{n}+\tfrac{ny}{n})\Big\}\Big),
\]
converges $H_{p}^{(n)}(\bar u+y,\bar\lambda)\to H_{p}(\bar u+y,\bar\lambda)$ to the solution of \eqref{Hp}  as $n\to \infty$.
Instead of giving tedious details in terms of the discrete version of \eqref{dHp}, we indicate below the key new argument in terms of continuous version of the integral $\xi\circ F(\cdot)$.
%\begin{align*}
%\sum_{j=1}^p\lambda_j\xi*F^{(n)}(\tfrac{k_j}{n})= \sum_{j=1}^p\Big(\xi(\tfrac{k_j}{n})F(0)+\sum_{i=1}^{k_j} \xi(\tfrac{k_j-i}{n})(F(\tfrac{i}{n})-F(\tfrac{i-1}{n}))\Big)
%\end{align*} 
%
%\begin{align*}
%\sum_{i=1}^p\lambda_i\xi*F^{(n)}(\tfrac{k_i+k}{n})&= \sum_{i=1}^p\lambda_i\sum_{j=0}^{k_i+k} \lambda_i\xi(\tfrac{k_i+k-j}{n})(F(\tfrac{j}{n})-F(\tfrac{j-1}{n}))\\
%&= \sum_{i=1}^p\lambda_i\sum_{j=0}^{k_i} \xi(\tfrac{k_i+k-j}{n})(F(\tfrac{j}{n})-F(\tfrac{j-1}{n}))
%+\sum_{j=1}^{k} \xi(\tfrac{k-j}{n})\sum_{i=1}^p\lambda_i(F(\tfrac{k_i+j}{n})-F(\tfrac{k_i+j-1}{n}))
%\end{align*} 
%
%\[\rE\Big(\exp\Big\{-\sum_{i=1}^p\lambda_i\sum_{j=0}^{k_i} \xi(\tfrac{k_i+k-j}{n})(F(\tfrac{j}{n})-F(\tfrac{j-1}{n}))\Big\}\,|\,\xi(\tfrac{0}{n}), \ldots,\xi(\tfrac{k}{n})\Big)
%=e^{-\xi(\tfrac{k}{n}) H_{p-1}^{(n)}(\bar u,\bar\lambda)}\]
%
%\[e^{-H_{p}^{(n)}(\bar u+\tfrac{k}{n},\bar\lambda)}=\rE_1\Big(\exp\Big\{-\xi(\tfrac{k}{n}) H_{p-1}^{(n)}(\bar u,\bar\lambda)
%-\xi*F_p^{(n)}(\tfrac{k}{n}))\Big\}\Big).
%\]
%
%\[F_p^{(n)}(y):=\sum_{i=1}^{p}\lambda_i (F^{(n)}(u_i+y)-F^{(n)}(u_i)),\quad y>0.\]
%

 Due to \eqref{dHp}, we have
\begin{align*}
  e^{-H_p(\bar u,\bar\lambda)} &=\rE_1\Big(\exp\Big\{-\sum\nolimits_{i=1}^p\lambda_i \xi\circ dF(u_i)\Big\}\Big),
 \end{align*} 
which in view of  \eqref{cbp} and \eqref{dHp}, yields
\begin{align*}
  e^{-H_p(\bar u+y,\bar\lambda)} &=\rE_1\Big(\exp\Big\{-\sum_{i=1}^p\lambda_i\int_0^{u_i+y}\xi(u_i+y-v)dF(v)\Big\}\Big).
 \end{align*} 
 Splitting each of the integrals in two parts $\int_0^{u_i+y}=\int_0^{u_i}+\int_{u_i}^{u_i+y}$, we find
\begin{align*}
  \sum_{i=1}^p\int_0^{u_i+y}\xi(u_i+y-v)dF(v)&= \sum_{i=1}^{p-1}\int_0^{u_i}\xi(u_i+y-v)dF(v)+\int_0^{y}\xi(y-v)dF^\circ_p(v),
 \end{align*} 
 and then, using the Markov property of the process $\xi(\cdot)$
\begin{align*}
 \rE\Big(\exp\Big\{- \sum_{i=1}^{p-1}\lambda_i\int_0^{u_i}\xi(u_i+y-v)dF(v)\Big\}\,&|\,\xi(u), 0\le u\le y\Big) 
% \\
% &= \rE_{\xi(y)}\Big(\exp\Big\{- \sum_{i=1}^{p-1}\lambda_i\xi_\gamma(u_i)\Big\}\Big) 
 = e^{-\xi(y)H_{p-1}(\bar u,\bar\lambda) },
 \end{align*} 
 we obtain
\begin{align*}
  e^{-H_p(\bar u+y,\bar\lambda)} &=\rE_1\Big(\exp\Big\{- \xi(y) H_{p-1}(\bar u,\bar\lambda)-\xi\circ F^\circ_p(y) \Big\}\Big)=\rE_1\Big(e^{-\xi\circ F_p(y)}\Big),
 \end{align*} 
 where $F_p(y):=H_{p-1}(\bar u,\bar\lambda)+F^\circ_p(y)$. After this, it remains to apply Proposition \ref{inteq}. 
\end{proof}

\section{Main results}\label{sec:lt}

The aim of this  chapter is to establish an fdd-convergence result for the vector $(X_1(\cdot),\ldots, X_q(\cdot))$ composed of the population counts 
corresponding to different individual scores $\chi_1(\cdot),\ldots,\chi_q(\cdot)$, which may depend on each other.% conditioned on $Z_0=n$ as $n\to\infty$.
%$$\{(X_1^{(n)}(t),\ldots, X_q^{(n)}(t)),\ t\ge0\}\stackrel{\rm fdd}{=}\{(X_1(t),\ldots, X_q(t)),\ t\ge0\,|\,Z_{0}=n\}$$
%Theorem \ref{nrta} in Section \ref{sec:cb} deals with the case of $q=1$ and $m_\chi<\infty$. Then, in Section \ref{sec:lt1} we consider the case of $q=1$ and $m_\chi=\infty$. Finally, Theorem \ref{nrtm} in Section \ref{lt3} takes up the general case with arbitrary $q$.

\subsection{Limit theorems}\label{sec:cb}

%The next result is a direct extension of the limit theorem \eqref{cgw} for Galton-Watson processes with non-overlapping generations. 
\begin{theorem}\label{nrt} Consider a population count defined by \eqref{X}.  If (\ref{crit}), $a<\infty$, $m_\chi<\infty$, then 
  \begin{equation}
    \label{nrta}
\{n^{-1}X(nu),\ u>0\,|\,Z_0=n\}\stackrel{\rm fdd\ }{\longrightarrow}\{m_\chi \xi(u a^{-1}),\ u>0\,|\,\xi(0)=1\},\ n\to\infty,
  \end{equation}
where $\xi(\cdot)$ is the continuous state branching process satisfying  \eqref{cbp}.
\end{theorem}

There are three new features in the limiting process of \eqref{nrta} compared to that of \eqref{cgw}:
\begin{description}
\item[ -] the continuous time parameter $u$ does not include zero, reflecting the fact that it may take some time for the distribution of ages of coexisting individuals to stabilise,
\item[ -] the time scale $a^{-1}$ corresponds to the scaling by the average length of overlapping generations, 
\item[ -] the factor $m_\chi$ accounts for the average $\chi$-score in a population with overlapping generations.
\end{description}

\begin{theorem}\label{nrtg} 
Consider a population count defined by \eqref{X}. Assume (\ref{crit}), $a<\infty$,  \eqref{gammo}, %In the case $m_\chi<\infty$, without loss of generality, put $\mathcal{L}(t)= m_\chi$ for all $t\ge1$.
and in the case $m_\chi=\infty$, assume additionally 
\begin{equation}\label{i}
    \rE(\chi^2(t))=o(t^{2\gamma}\mathcal{L}^2(t)),\quad t\to\infty.
\end{equation}
 Then %the weak convergence of fdd holds
  \begin{equation}
    \label{nrtb}
\{n^{-1-\gamma}\mathcal{L}^{-1}(n)X(nu),\ u>0\,|\,Z_0=n\}\stackrel{ \rm fdd\ }{\longrightarrow}\{a^{\gamma-1}\xi_\gamma(ua^{-1}),\ u>0\,|\,\xi(0)=1\},\quad n\to\infty,
  \end{equation}
where $\xi_\gamma(\cdot)$ is given by  \eqref{asmaa}, which is understood according to the previous chapter.
%where $\xi_0(y)=\xi(y)$ is the continuous state branching process satisfying (\ref{cbp}) with $\xi(0)=1$ and $\xi_\gamma(y)=\int_0^y\xi(y-u)du^\gamma$ for $\gamma>0$.
\end{theorem}

The next result extends Theorems \ref{nrt} and \ref{nrtg} to the case of several population counts.  

\begin{theorem}\label{nrtm} Consider $q\ge1$ population counts $X_1(t),\ldots,X_q(t)$, each defined by Definition \ref{def1} in terms of different individual scores $\chi_1(t),\ldots,\chi_q(t)$. Assume (\ref{crit}), $a<\infty$, and  \eqref{gammo}, with $\gamma=\gamma_j$ and $\mathcal{L}=\mathcal{L}_j$ for the $\chi_j$-score, $j=1,\ldots,q$.  If $m_{\chi_j}=\infty$, assume additionally condition \eqref{i} for the $\chi_j$-score. 

Then, as $n\to\infty$, %the weak convergence of fdd holds
  \[
\Big({X_1(nu)\over n^{1+\gamma_1}\mathcal{L}_1(n)},\ldots,{X_q(nu)\over n^{1+\gamma_q}\mathcal{L}_q(n)}\,|\,Z_0=n\Big)_{u>0}\stackrel{\rm fdd\ }{\longrightarrow}\Big(a^{\gamma_1-1}\xi_{\gamma_1}(ua^{-1}),\ldots, a^{\gamma_q-1}\xi_{\gamma_q}(ua^{-1})\,|\,\xi(0)=1\Big)_{u>0}.
  \]
\end{theorem}

To illustrate the utility of Theorem \ref{nrtm}, we consider a multitype GW-process 
$$\{(Z^1_t,Z^2_t,\ldots,Z^q_t), t\ge0\,|\,Z^1_0=n\},$$ 
where $Z^i_t$ is the number of type $i$-individuals born at time $t$, for $i=1,\ldots,q$. Each individual of type $i$ is assumed to live one unit of time and then be replaced by $N_{ij}$ individuals of type $j$. 
Denoting $m_{ij}:=\rE(N_{ij})$, assume that the multitype GW- process is decomposable in that 
\begin{align}
m_{ij}=0,\quad 1\le j<i\le q.\label{md}
\end{align}
The next result deals with a decomposable critical GW-process, satisfying
\begin{align}
m_{jj}=1,\quad 1\le j\le q,\qquad m_{j-1,j}\in(0,\infty),\quad 2\le j\le q,\qquad m_{ij}<\infty,\quad 1\le i \le j\le q. \label{mcri}
\end{align}

To put this process into the GWO-framework, we treat as GWO-individuals only the type 1 individuals, while the other types will be addressed by respective population counts. Clearly, the numbers of GWO-individuals forms a single type GW-process, and \eqref{cgw}, derived from Corollary 1, describes the limit behaviour of the scaled process $(Z^1_t, t\ge0|Z^1_0=n)$.
Since the process $\{Z^1_0,\ldots,Z^1_{n-1}\,|\,Z^1_0=n\}$ during $n$ units of time, produces type 2 individuals, of order $n$ new individuals per unit of time, one would expect, in view of Theorem \ref{nrtm}, a typical number of type 2 individuals at time $n$ to be of order $n^2$. An extrapolation of this reasoning suggests scaling by $n^j$ for the number of type $j$ individuals, $j=1,\ldots,q$.

\begin{theorem}\label{deco} 
Consider a decomposable multitype GW-process $(Z_t^1,Z^2_t,\ldots,Z^q_t)$ starting with $n$ individuals of type 1. Assume \eqref{md} and \eqref{mcri}.  
If furthermore,
$\rV(N_{jj})<\infty$, for all $1\le j\le q$, and $\rV(N_{11})=2b$, then  %the weak convergence in $D[0,\infty)$
  \begin{equation*}
   % \label{dgw}
\{(n^{-1}Z^1_{ny},n^{-2}Z^2_{ny},\ldots,n^{-q}Z^q_{ny}),\ y\ge0\,|\,Z^1_0=n\}\stackrel{\rm fdd\ }{\longrightarrow}\{(\xi(y),\alpha_1\xi_1(y),\ldots,\alpha_{q-1}\xi_{q-1}(y))\ y\ge0\,|\,\xi(0)=1\}
  \end{equation*}
as $n\to\infty$, with 
$\alpha_{j}:=\tfrac{1}{j!}m_{1,2}\cdots m_{j,j+1},\quad j=1,\ldots,q-1.$
\end{theorem}

Here the limiting process $\xi(\cdot)$ is the same as in \eqref{cgw} and
$\xi_j(y)=\int_0^y\xi(y-u)du^j$.
Notice that the only source of randomness in the $q$-dimensional limit process is due to the randomly fluctuating number of the first type of individuals. Observe also, that only the means $m_{j,j+1}$ appear in the limit, but not the other means like for example $m_{1,3}$. This fact reflects the following phenomenon of the reproduction system under consideration: in a large population,  the number of type 3 individuals stemming directly from type 1 individuals is much smaller compared to the number of type 3 individuals stemming from type 2 individuals.

\subsection{Proof of Theorem \ref{nrt}}

Assuming \eqref{ula}, put
\begin{equation}\label{Lnp}
\Lambda_{n,p}(t):=\ln  \rE_1\Big(\exp\Big\{-n^{-1}\sum_{i=1}^p\lambda_iX(nu_i+ t)\Big\}\Big).
\end{equation}
Due to \eqref{layla}, the Laplace transform of the $p$-dimensional distributions of the scaled $X(\cdot)$ are given by
\[\rE_n\Big(\exp\Big\{-n^{-1}\sum_{i=1}^p\lambda_iX(n(u_i+ y))\Big\}\Big)=e^{-n\Lambda_{n,p}(ny)},\quad y\ge0.\]
We prove Theorem \ref{nrt}  by  showing that 
\begin{equation}
  \label{lambdap}
 n\Lambda_{n,p}(ny)\stackrel{y}{\to}r_p(y),\quad n\to\infty,
\end{equation}
where the function
$r_p(y):=G_p(a^{-1}(\bar u+y),m_\chi\bar\lambda)$
determines the limiting fdd's of Theorem \ref{nrt} through Lemma \ref{lla}. Our proof of \eqref{lambdap} consists of several steps summarised in the next flow chart.
\begin{align}
\eqref{lambdap}\longleftarrow\eqref{nBn}\longleftarrow
\left.
\begin{array}{cc}
  \eqref{bound} & \longleftarrow\\
 \eqref{nBn1} &\\
\eqref{nBn2}  &  \longleftarrow
\end{array}
\right.
\left.
\begin{array}{l}
  \eqref{drug} \\
\\
  \eqref{astan}, \eqref{Nbn23}, \eqref{Nbn22} 
\end{array}
\right.
\label{kairat}
\end{align}

Due to Lemma \ref{lel}, we have
\begin{equation}
  \label{chuk}
\Lambda_{n,p}(t)=\Lambda^{[\psi_{n,p}]}(t),
\end{equation}
with
\begin{equation}
  \label{psina}
\psi_{n,p}(t)= n^{-1} \sum\nolimits_{i=1}^p\lambda_i\chi(nu_i+t).
\end{equation}
On the other hand, according to (\ref{Gp}), the limit function $r_p(\cdot)$ satisfies
\begin{equation}
  \label{Gpa}
 r_p(y)= r_{p-1}(0)+\lambda_p m_\chi-ba^{-1}\int_0^yr_p^2(v)dv.
\end{equation}
%where in accordance with our notational agreement,
%\begin{align*}
%r_{p-1}(0)&=G_{p-1}(a^{-1}\bar u,m_\chi\bar\lambda)=G_{p-1}(a^{-1}u_1,\ldots, a^{-1}u_{p-1}; m_\chi\lambda_1,,\ldots, m_\chi\lambda_{p-1}).
%\end{align*}
Thus, we can prove relation (\ref{lambdap}) using Proposition \ref{cth} 
%\[  e^{-\Lambda_{n,p}(t)}=\rE\Big(\exp\Big\{-n^{-1}\sum_{i=1}^p\lambda_iX(nu_i+ t)\Big\}\Big)\]
%
%\begin{equation}\label{psinp}
%  \Lambda_{n,p}(t)=
%\end{equation}
%with $\lambda_{n,i}=n^{-1}\lambda_i$. In terms of the Laplace transform
%$$e^{-n\Lambda_{n,p}(t)}=E\left(e^{-X^{(n)}_{\psi_{n,p}}(t)}\right)= \rE\left(\exp\left\{-n^{-1}\sum_{i=1}^p\lambda_iX^{(n)}(nu_i+ t)\right\}\right)$$ 
%fully characterizes the $p$-dimensional distributions of the process $\{n^{-1}X^{(n)}(ny),\ y>0\}$,
and induction over $p$ by verifying that
\begin{equation}
  \label{nBn}
 nB_n(ny)\stackrel{y}{\to}r_{p-1}(0)+\lambda_p m_\chi,
\end{equation}
where in accordance with \eqref{dy} and  \eqref{by},
\begin{equation}
  \label{bnt}
B_n(t)=e_2^{\Lambda_{n,p}(t)}+\sum\nolimits_{t=1}^\infty e_1^{\Lambda_{n,p}(-t)}R_{ny}(t)+D_n*U(t)
\end{equation}
and
\begin{equation}
  \label{dnt}
D_n(t)=\rE_1\Big(e_1^{\psi_{n,p}(t)}e^{-\sum_{j=1}^\infty\Lambda_{n,p}(t-j)\nu(j)}\Big).
\end{equation}

The initial induction step, with $p=0$, becomes trivial if we set $r_0(y):=0$ for all $y$. 
To state a relevant induction assumption, denote 
\begin{equation}\label{Lnp'}
\Lambda'_{n,p-1}(t):=\ln  \rE_1\Big(\exp\Big\{-n^{-1}\sum\nolimits_{i=1}^{p-1}\lambda_iX(nu'_i+ t)\Big\}\Big),
\end{equation}
where $u'_1>u'_2>\ldots>u'_{p-1}$ and $\lambda_1\ge0,\ldots, \lambda_{p-1}\ge0$.
Then, the inductive hypothesis claims
\begin{equation}
  \label{la'}
 n\Lambda'_{n,p-1}(ny)\stackrel{y}{\to}G_{p-1}(a^{-1}(\bar u'+y),m_\chi \bar\lambda),\quad n\to\infty.
\end{equation}
%stating (\ref{lambdap}) with $p-1$ instead of $p$ implies the following relation
%\begin{equation}
%  \label{lambdap-1}
% \sup_{0\le 2v\le u_{p-1}}\,\left|\,n\Lambda_{n,p-1}(-nv)-G_{p-1}(a^{-1}(\bar u-v),m_\chi \bar\lambda)\,\right|\,\to0,\ n\to\infty,
%\end{equation}
%which will be used soon in deducing  (\ref{lambdap}). 
% Stating the inductive hypothesis 
%\begin{equation}
%  \label{lambdap-1}
% n\Lambda'_{n,p-1}(ny)\stackrel{y\,}{\Rightarrow} r'_{p-1}(y),\quad n\to\infty,
% \end{equation}
%we may refer to a stronger form of uniform convergence since by \eqref{ula} 
%\[n(u_i+y)\stackrel{y\,}{\Rightarrow} \infty,\quad i=1,\ldots,p-1.\]
We establish the uniform convergence  (\ref{nBn}), under assumption   \eqref{la'}, in three steps
\begin{align}
ne_2^{\Lambda_{n,p}(ny)} \stackrel{y\,}{\Rightarrow}0,\quad n\to\infty,& \label{bound}\\
n\sum\nolimits_{t=1}^\infty e_1^{\Lambda_{n,p}(-t)}R_{ny}(t)\stackrel{y}{\to}r_{p-1}(0),\quad n\to\infty, &  \label{nBn1}\\
n \sum\nolimits_{t=1}^{ny}D_n(ny-t)U(t)\stackrel{y}{\to}\lambda_p m_\chi,\quad n\to\infty.  & \label{nBn2}
\end{align}

\begin{proof} of  \eqref{bound}.  
The upper bound
\[
e_1^{n\Lambda_{n,p}(t)}\le n\rE_1(X^{[\psi_{n,p}]}(t))=\sum\nolimits_{i=1}^p\lambda_{i}\rE_1(X(nu_i+t)),
\]
under the assumption $m_\chi<\infty$, implies
\begin{equation}
  \label{drug}
\sup_{n\ge1}\sup_{-\infty<t\le ny}n\Lambda_{n,p}(t)<\infty\text{ for any }y>0.
\end{equation}
This and a corollary of \eqref{e12},
$ne_2^{\Lambda_{n,p}(ny)}\le \tfrac{n}{2}\Lambda_{n,p}^2(ny)$,
 entail (\ref{bound}).
\end{proof}

\begin{proof} of  \eqref{nBn1}. Setting $u'_i:=u_i-u_{p-1}$,  recall  \eqref{Lnp'}. Notice that since $u_p=0$, we get for $t>0$,
\begin{align*}
 \Lambda_{n,p}(-t)&=\ln  \rE_1\Big(\exp\Big\{-n^{-1}\sum_{i=1}^{p-1}\lambda_iX(nu_i-t)\Big\}\Big)\\
 &=\ln  \rE_1\Big(\exp\Big\{-n^{-1}\sum_{i=1}^{p-1}\lambda_iX(n(u'_i+u_{p-1})-t)\Big\}\Big)=\Lambda'_{n,p-1}(nu_{p-1}-t).
\end{align*}
By the inductive assumption \eqref{la'}, the function
$$r^{(n)}(t):= ne_1^{n\Lambda'_{n,p-1}(nu_{p-1}-t)}1_{\{1\le t\le nu_{p-1}/2\}}$$
satisfies
  \[
r^{(n)}(ny)\stackrel{y\,}{\Rightarrow}r(y),\ n\to\infty,\quad r(y):=G_{p-1}(a^{-1}(\bar u-y),m_\chi \bar\lambda)1_{\{0\le y\le u_{p-1}/2\}}.
\]
Moreover, due to \eqref{drug}, we have 
$0\le r^{(n)}(t)\le C$ for all $n, t\ge1.$
Since $r(0)=r_{p-1}(0)$, relation (\ref{nBn1}) now follows from  Lemma \ref{rty}.
\end{proof}

\begin{proof} of  (\ref{nBn2}). In view of
\begin{align*}
 D_n(t)%&=\rE\Big(e_1^{\psi_{n,p}(t)}e^{-\sum_{j=1}^\infty\Lambda_{n,p}(t-j)\nu(j)})\Big)
 =\rE(\psi_{n,p}(t))-\rE(e_2^{\psi_{n,p}(t)})-\rE\Big(e_1^{\psi_{n,p}(t)}e_1^{\sum_{j=1}^\infty\Lambda_{n,p}(t-j)\nu(j)}\Big),
\end{align*}
relation (\ref{nBn2}) follows from \eqref{ert}  and the next three relations
 \begin{align}
n\sum\nolimits_{t=1}^{ny}\rE(\psi_{n,p}(ny-t))U(t)\stackrel{y}{\to} \lambda_p m_\chi,\quad n\to\infty,& \label{astan}\\
    n\sum\nolimits_{t=1}^{ny}\rE(e_2^{\psi_{n,p}(t)})\stackrel{y\,}{\Rightarrow}0,\quad n\to\infty,& \label{Nbn23}\\
     n \sum\nolimits_{t=1}^{ny}\rE\Big(e_1^{\psi_{n,p}(t)}e_1^{\sum_{j=1}^\infty\Lambda_{n,p}(t-j)\nu(j)}\Big)\stackrel{y\,}{\Rightarrow}0,\quad n\to\infty.& \label{Nbn22}
\end{align}

To prove of \eqref{astan}, notice, that \eqref{psina} gives 
$n\psi_{n,p}(t)=\lambda_p\chi(t)+n\psi_{n,p-1}(t)$.
Since
\[\sum\nolimits_{t=1}^{ny}\rE(\chi(ny-t))U(t)\stackrel{y}{\to} m_\chi,\quad n\to\infty,\]
it suffices to check that
 \begin{align}
  n\sum\nolimits_{t=1}^{ny}\rE(\psi_{n,p-1}(t))\stackrel{y\,}{\Rightarrow} 0,\quad n\to\infty. \label{nBn21}
\end{align}
This follows from the fact that for any positive $u$,
\[\sum\nolimits_{t=1}^{ny}\rE(\chi(nu+t))\le \sum\nolimits_{t>nu}\rE(\chi(t)),\]
with the right-hand side going to 0 as $n\to\infty$ under the assumption $m_\chi<\infty$.

Turning to  (\ref{Nbn23}), we split its left-hand side in three parts using (\ref{e1}), and then produce an upper bound as a sum of three terms involving an arbitrary $k\ge 1$:
\begin{eqnarray*}
 ne_2^{\psi_{n,p}(t)} &=& ne_2^{n^{-1}\lambda_{p}\chi(t)}+ne_2^{\psi_{n,p-1}(t)}+ne_1^{n^{-1}\lambda_{p}\chi(t)}e_1^{\psi_{n,p-1}(t)} \\
  &\le& n^{-1}\lambda_{p}^2 \chi^2(t) 1_{\{\chi(t)\le k\}}+\lambda_{p}\chi(t)1_{\{\chi(t)> k\}}+2n\psi_{n,p-1}(t).
\end{eqnarray*}
 The third term is handled by (\ref{nBn21}). The first term is further estimated from above by 
 $$n^{-1} \sum\nolimits_{t=1}^{ny}\rE(\chi^2(t) 1_{\{\chi(t)\le k\}}) \le n^{-1}k \sum\nolimits_{t=1}^{\infty}\rE(\chi(t)),$$
where the right-hand side converges to zero for any fixed $k$. Finally, in view of
 $$\sum\nolimits_{t=1}^{ny}\rE(\chi(t) 1_{\{\chi(t)> k\}}) \le \sum\nolimits_{t=1}^{\infty}\rE(\chi(t) 1_{\{\chi(t)> k\}}),$$
the proof of  (\ref{Nbn23}) is finished by applying Fatou's lemma as $k\to\infty$. 

To prove convergence (\ref{Nbn22}), we use the bound 
$$e_1^{\psi_{n,p}(t)}\le n^{-1} \lambda_{p}\chi(t)+\psi_{n,p-1}(t),$$ 
and referring to  (\ref{nBn21}), reduce the task to
\begin{equation*}%\label{chiX}
  \sum\nolimits_{t=1}^{ny}\rE\Big(\chi(t)e_1^{\sum_{j=1}^\infty\Lambda_{n,p}(t-j)\nu(j)}\Big)\stackrel{y\,}{\Rightarrow}0,\quad n\to\infty.
\end{equation*}
The last relation follows from the upper bound
\begin{align*}
 \sum_{t=1}^{\infty}\rE\Big(\chi(t)e_1^{\sum_{j=1}^\infty\Lambda_{n,p}(t-j)\nu(j)}\Big)&\le  \sum_{t=k}^{\infty}\rE(\chi(t))+\sum_{t=1}^{k}\rE(\chi(t)1_{\{\chi(t)>k_1\}}) +k_1\sum_{t=1}^{k}\sum_{j=1}^\infty\Lambda_{n,p}(t-j)A(j)
 %\\&\le\sum_{t=k}^{\infty}m(t)+\sum_{t=1}^{\infty}E\left(\chi(t)1_{\{\chi(t)>k_1\}}\right)+Ckk_1\sup_{0\le t\le k}E\left(X_{\psi_{n,p}}(t)\right)
\end{align*}  
because the third term tends to 0 as $n\to\infty$, thanks to (\ref{drug}),  and the first two terms in the right-hand side vanish as $k\to\infty$ and $k_1\to\infty$ due to the assumption $m_\chi<\infty$.
\end{proof}

\subsection{Proof of Theorem \ref{nrtg}}
The main idea of the proof of Theorem \ref{nrtg} is the same as of Theorem \ref{nrt}, and here we mainly focus on the new argument addressing the case $m_\chi=\infty$. We want  to prove \eqref{lambdap} with the modified annotations
\begin{align*}%\label{Lnp}
\Lambda_{n,p}(t)&:=\ln  \rE_1\Big(\exp\Big\{-n^{-1-\gamma}\mathcal{L}^{-1}(n)\sum_{i=1}^p\lambda_iX(nu_i+ t)\Big\}\Big),\\
r_p(y)&:=H_p(a^{-1}(\bar u+y),a^{\gamma-1}\bar \lambda),
\end{align*}
where $H_p(\bar u,\bar \lambda)$ is defined by  \eqref{dHp}, with $F(y):=y^\gamma$.
In this case, relation \eqref{chuk} holds with
\[\psi_{n,p}(t):=\sum\nolimits_{i=1}^p\lambda_{n,i}\chi(nu_i+t),\quad \lambda_{n,i}:=\lambda_i\, n^{-1-\gamma}\mathcal{L}^{-1}(n),\]
and according to \eqref{Hp}, the right-hand side of \eqref{lambdap} satisfies
\[ r_p(y)= r_{p-1}(0)+a^{-1}\sum_{i=1}^p\lambda_i\left((u_i+y)^\gamma-u_i^\gamma\right)-ba^{-1}\int_0^yr_p^2(v)dv.\]

Thus, under the conditions of Theorem \ref{nrtg}, relation \eqref{lambdap} will follow from Proposition \ref{cth} after we show 
\[nB_n(ny)\stackrel{y}{\to}r_{p-1}(0)+a^{-1}\sum\nolimits_{i=1}^{p}\lambda_i\left((u_i+y)^\gamma+u_i^\gamma\right),\quad n\to\infty,\]
where $B_n(t)$ is defined by \eqref{bnt} and \eqref{dnt}. Its counterpart \eqref{nBn} was proven in the case $m_\chi<\infty$ according to flow chart \eqref{kairat}. In the rest of the proof, we follow the same flow chart and comment on necessary changes in the case $m_\chi=\infty$.

 The counterparts of \eqref{bound}. \eqref{drug}, and \eqref{nBn1} in  the case $m_\chi=\infty$, are verified in a similar way as   in  the case $m_\chi<\infty$, now using Proposition \ref{rt}. 
 The counterpart of \eqref{nBn2} takes the form
\begin{equation}
  \label{Nbn2}
n\sum\nolimits_{t=1}^{ny}D_n(ny-t)U(t)\stackrel{y}{\to}a^{-1}\sum\nolimits_{i=1}^{p}\lambda_i\left((u_i+y)^\gamma-u_i^\gamma\right),\quad n\to\infty,
\end{equation}
as Proposition \ref{rt} yields the following counterpart of  \eqref{astan}
 \begin{align*}
n\sum\nolimits_{t=1}^{ny} \rE(\psi_{n,p}(ny-t))U(t)&\stackrel{y}{\to} a^{-1}\sum\nolimits_{i=1}^{p}\lambda_i\left((u_i+y)^\gamma-u_i^\gamma\right),\quad n\to\infty.  %\label{Nbn21}
\end{align*}

To verify \eqref{Nbn23} in the case $m_\chi=\infty$, we check that
\begin{equation}
  \label{CS}
n\sum\nolimits_{t=1}^{ny}\rE(\psi_{n,p}^2(t))\stackrel{y\,}{\Rightarrow}0,\quad n\to\infty,
\end{equation}
by putting to use condition \eqref{i} to handle the terms
$$ n\sum_{i=1}^p\sum_{t=1}^{ny}\lambda_{n,i}^2\rE(\chi^2(nu_i+t))+2n\sum_{i=1}^p\sum_{j=i+1}^p\sum_{t=1}^{ny}\lambda_{n,i}\lambda_{n,j}\rE(\chi(nu_i+t)\chi(nu_j+t)),$$ 
after applying the Cauchy-Schwartz inequality for expectations
$$\rE(\chi(nu_i+t)\chi(nu_j+t))\le \sqrt{\rE(\chi^2(nu_i+t))} \sqrt{\rE(\chi^2(nu_i+t))}.$$ 

To prove the counterpart of \eqref{Nbn22} in the case $m_\chi=\infty$, we use a sequence of upper bounds
\begin{align*} 
n \sum_{t=1}^{ny}\rE\Big(e_1^{\psi_{n,p}(t)}&e_1^{\sum_{j=1}^\infty\Lambda_{n,p}(t-j)\nu(j)}\Big)\\
&\le n \sum_{t=1}^{ny}\rE(\psi_{n,p}(t)1_{\{N> n\epsilon\}})
+n\sum_{t=1}^{ny}\rE\Big(\psi_{n,p}(t)\sum_{j=1}^\infty\Lambda_{n,p}(t-j)\nu(j)1_{\{N\le n\epsilon\}}\Big)\\
&\le n \sum_{t=1}^{ny}\sqrt{\rE(\psi^2_{n,p}(t))}\sqrt{\rP(N> n\epsilon)}
+\sup_{t\le ny}(n\Lambda_{n,p}(t))\sum_{t=1}^{ny}\rE(\psi_{n,p}(t)N1_{\{N\le n\epsilon\}})\\
&\le C_1 \epsilon^{-1}\sum_{t=1}^{ny}\sqrt{\rE(\psi^2_{n,p}(t))}
+C_2n\epsilon\sum_{t=1}^{ny}\rE(\psi_{n,p}(t)),
\end{align*}  
where we applied the Cauchy-Schwartz and Markov inequalities together with  \eqref{drug}. By the Cauchy-Schwartz inequality  for the dot product
\[\Big(\sum_{t=1}^{ny}1\cdot \sqrt{\rE(\psi^2_{n,p}(t))}\ \Big)^2\le ny\sum_{t=1}^{ny}\rE(\psi^2_{n,p}(t)), %=n^{-1-\gamma}\mathcal{L}^{-1}(n)\sum_{t=1}^{ny}\sqrt{\rE\Big(\Big(\sum_{i=1}^p\lambda_{i}\chi(nu_i+t)\Big)^2\Big)}\stackrel{y\,}{\Rightarrow}0,\quad n\to\infty.
\]
which together with \eqref{CS} yield
$\sum_{t=1}^{ny}\sqrt{\rE(\psi^2_{n,p}(t))}\stackrel{y\,}{\Rightarrow}0$ as $n\to\infty$.
On the other hand, in view of Proposition \ref{rt}, the upper bound
\[n\epsilon\sum\nolimits_{t=1}^{ny}\rE(\psi_{n,p}(t))<\epsilon C(y_1),\quad n\ge 1, \quad 0\le y\le y_1\]
holds for an arbitrary $\epsilon>0$.
Sending $\epsilon\to0$ ends the proof of \eqref{Nbn22} and thereby of Theorem \ref{nrtg}.
%\end{proof}

 \subsection{Proof of Theorem \ref{nrtm}}\label{sec:lt3}
 
\begin{lemma}
Put
\[H_{p,q}(\bar u,\bar\lambda)
:=-\ln\rE_1\Big(\exp\Big\{-\sum_{i=1}^p\sum_{j=1}^q\lambda_{ij}\xi_{\gamma_j}(u_i)\Big\}\Big),\]
assuming
\begin{equation}\label{gj}
    0=\gamma_1=\ldots=\gamma_s<\gamma_{s+1}\le\ldots\le\gamma_q,\quad 0\le s\le q.
\end{equation}
Then for $u_1>\ldots>u_p=0$, the following integral equation holds
\begin{align}
  \label{Hpq}
H_{p,q}(\bar u+y,\bar\lambda)&= H_{p-1,q}(\bar u,\bar\lambda)+F_{p,q}(y)-b\int_0^yH_{p,q}^2(\bar u+v,\bar\lambda)dv,\\
F_{p,q}(y)&:=\lambda_{p1}+\ldots+\lambda_{ps}+ \sum_{i=1}^p\sum_{j=s+1}^q\lambda_{ij}\left((u_i+y)^{\gamma_j}-u_i^{\gamma_j}\right).\nonumber
\end{align}
\end{lemma}

\begin{proof} The lemma is proven similarly to Lemma \ref{lilla}. 
\end{proof}
Theorem \ref{nrtm} is obtained by combining the proofs of Theorems \ref{nrt} and \ref{nrtg}. The aim is to prove \eqref{lambdap} with
\begin{align*}
 \psi_{n,p}(t)&:=\sum_{i=1}^p\sum_{j=1}^q\lambda_{n,ij}\chi_j(nu_i+t),\ \lambda_{n,ij}:=\lambda_{ij}n^{-1-\gamma_j}\mathcal{L}_j^{-1}(n),\\
 r_p(y)&:=H_{p,q}(a^{-1}(\bar u+y),a^{\gamma_1-1}\bar \lambda_1,\ldots,a^{\gamma_q-1}\bar \lambda_q),
\end{align*}
assuming $u_1>\ldots>u_{p-1}>u_p=0$ and $\lambda_{ij}\ge0$.
Without loss of generality we assume \eqref{gj} and that for some $0\le s'\le s$,
$$m_{\chi_j}<\infty,\quad j=1,\ldots,s',\qquad m_{\chi_j}=\infty,\quad j= s'+1,\ldots, q.$$ 

According to (\ref{Hpq}), the limit function in \eqref{lambdap} satisfies the integral equation
\[ r_p(y)= r_{p-1}(0)+a^{-1}F_{p,q}(y)-ba^{-1}\int_0^yr_p^2(v)dv.\]
Therefore, to apply Proposition \ref{cth}, we have to prove for the updated version of \eqref{bnt}, that
\[nB_n(ny)\stackrel{y}{\to}r_{p-1}(0)+a^{-1}F_{p,q}(y), \quad n\to\infty,\]
which once again, is done  according to flow chart \eqref{kairat}.
Even in this more general setting, the counterparts of \eqref{bound} and \eqref{nBn1} are valid, and the task boils down to verifying the counterpart of \eqref{nBn2} 
\[
 n \sum\nolimits_{t=1}^{ny}D_n(ny-t)U(t)\stackrel{y}{\to}a^{-1}F_{p,q}(y), \quad n\to\infty,
\]
where  the limit is obtained using Proposition \ref{rt} for the counterpart of \eqref{astan}
 \[n\sum\nolimits_{t=1}^{ny} \rE(\psi_{n,p}(ny-t))U(t)\stackrel{y}{\to} a^{-1}F_{p,q}(y), \quad n\to\infty.\]
 
It remains to verify the counterparts of \eqref{Nbn23}, (\ref{Nbn22}).

\begin{proof} {\sc of} (\ref{Nbn23}).  Observe that 
$\psi_{n,p}(t)=\psi'_{n,p}(t)+\psi''_{n,p}(t)$, where
\[\psi'_{n,p}(t):=\sum_{j=1}^{s'}\sum_{i=1}^p\lambda_{n,ij}\chi_j(nu_i+t),\quad \psi''_{n,p}(t):=\sum_{j=s'+1}^q\sum_{i=1}^p\lambda_{n,ij}\chi_j(nu_i+t).\]
Using (\ref{e1}), 
we can split the left-hand side of (\ref{Nbn23}) into the sum of three terms 
\begin{eqnarray*}
   n\sum_{t=1}^{ny}\rE(e_2^{\psi_{n,p}(t)})= n\sum_{t=1}^{ny}\rE(e_2^{\psi'_{n,p}(t)})+  n\sum_{t=1}^{ny}\rE(e_2^{\psi''_{n,p}(t)})+ n\sum_{t=1}^{ny}\rE(e_1^{\psi'_{n,p}(t)} e_1^{\psi''_{n,p}(t)}).
\end{eqnarray*}
The first and the second terms  are handled using the argument of the proofs of Theorems \ref{nrt} and \ref{nrtg} respectively. 

The third term requires a special attention. It is estimated from above by
\begin{align*}
 n\sum_{t=1}^{ny}\rE(e_1^{\psi'_{n,p}(t)} e_1^{\psi''_{n,p}(t)})
&\le  n\sum_{t=1}^{ny}\rE \Big(\Big(e_1^{\sum_{j=1}^{s'}\lambda_{n,pj}\chi_j(t)}+ e_1^{\psi'_{n,p-1}(t)}\Big)e_1^{\psi''_{n,p}(t)}\Big)\\
&\le  C\sum_{j=1}^{s'}\sum_{t=1}^{ny}\rE \Big(\chi_j(t) e_1^{\psi''_{n,p}(t)}\Big)+n\sum_{t=1}^{ny}\rE (\psi'_{n,p-1}(t)).
\end{align*}
The last term is tackled in a similar way as \eqref{nBn21}, and it remains to show that for each $j\le s'$,
\[\sum_{t=1}^{ny}\rE \Big(\chi_j(t) e_1^{\psi''_{n,p}(t)}\Big)\stackrel{y\,}{\Rightarrow}0, \quad n\to\infty.\]
To this end, observe that for an arbitrary $k\ge1$,
\[\sum_{t=1}^{ny}\rE \Big(\chi_j(t) e_1^{\psi''_{n,p}(t)}\Big)\le k\sum_{t=1}^{ny}\rE (\psi''_{n,p}(t))+\sum_{t=1}^{\infty}\rE (\chi_j(t) 1_{\{\chi_j(t)>k \}}).\]
The first term is taken care by \eqref{nBn21}, while the second term vanishes as $k\to\infty$ since $m_{\chi_j}<\infty$.
\end{proof}

\begin{proof} {\sc of}  \eqref{Nbn22}. 
Using $\psi_{n,p}(t)=\psi'_{n,p}(t)+\psi''_{n,p}(t)$, we get
$e_1^{\psi_{n,p}(t)}\le e_1^{\psi'_{n,p}(t)}+e_1^{\psi''_{n,p}(t)}$,
which allows us to replace \eqref{Nbn22} by the following two relations
\begin{align*}
  n \sum_{t=1}^{ny}\rE\Big(e_1^{\psi'_{n,p}(t)}e_1^{\sum_{j=1}^\infty\Lambda_{n,p}(t-j)\nu(j)}\Big)&\stackrel{y\,}{\Rightarrow}0,\quad n\to\infty,\\
   n \sum_{t=1}^{ny}\rE\Big(e_1^{\psi''_{n,p}(t)}e_1^{\sum_{j=1}^\infty\Lambda_{n,p}(t-j)\nu(j)}\Big)&\stackrel{y\,}{\Rightarrow}0,\quad n\to\infty.
\end{align*}
The first relation is proven in the same way as \eqref{Nbn22} was proven for Theorem \ref{nrt}, and the second relation is proven in the same way as \eqref{Nbn22} was proven for Theorem \ref{nrtg}.
\end{proof}

 \subsection{Proof of Theorem \ref{deco}}\label{seco}
Adapting the setting of Theorem \ref{deco} to Theorem \ref{nrtm}, we treat the process $(Z_{t}^1,\ldots,Z_{t}^q)$ as a vector of population counts for a single type GW-process. This is achieved by focussing on the type 1 individuals and introducing $q$ individual scores for a generic individual of type 1 born at time 0 by setting 
\begin{description}
\item[ ] $\chi_1(t):=1_{\{t=0\}}$,
\item[ ] $\chi_j(t):=$ the number of descendants of the generic individual, which (a) have no other intermediate ancestors of type 1, (b) are born at time $t$,  and (c) have type $i$,  
\end{description}
for $j=2,\ldots,q$. Having this, our task  is to check that conditions \eqref{gammo} and (\ref{i}) hold  with 
$\gamma:=j-1$ and $\mathcal{L}(t):=\tfrac{1}{(j-1)!} m_{1,2}\cdots m_{j-1,j}$ for all $t\ge1$ and $j=2,\ldots,q$.

To check condition \eqref{gammo} with  $\chi(\cdot):=\chi_j(\cdot)$ for a given $j=2,\ldots,q$, we use a representation
\begin{equation}\label{aux}
   \chi_j(t+1)=\sum_{i=2}^j \sum_{k=1}^{N_{1i}}Z^{i}_t(j,k),
\end{equation}
where $Z^{i}_t(j,k)\stackrel{\rm d}{=}Z^{i}_t(j)$ stands for the number of type $j$ individuals born at time $t+1$ and descending from a  type $i$ individual born at time 1. 
This gives
\[   \rE(\chi_j(t+1))=\sum\nolimits_{i=2}^j m_{1j}M_{ij}(t),\quad M_{ij}(t):=\rE(Z^{i}_t(j)|Z^i_0=1).
\]
Furthermore, due to the decomposable  branching property, we have
\begin{equation}\label{rau}
Z^{i}_{t+1}(j)=\sum\nolimits_{l=i}^j \sum\nolimits_{k=1}^{N_{il}}Z^{l}_t(j,k),
\end{equation}
implying the recursion
\[M_{ij}(t+1)=\sum\nolimits_{l=i}^j m_{il}M_{lj}(t)=M_{ij}(t)+\sum\nolimits_{l=i+1}^{j-1}m_{il}M_{lj}(t)+m_{ij}.\]
Putting here $j=i+1$, we get $M_{i,i+1}(t)=t m_{i,i+1}$. From 
\[M_{i,i+2}(t+1)=M_{i,i+2}(t)+m_{i,i+1}M_{i+1,i+2}(t)+m_{i,i+2}=M_{i,i+2}(t)+t m_{i,i+1} m_{i+1,i+2}+m_{i,i+2},\]
we find
$M_{i,i+2}(t)\sim \tfrac{1}{2}t^2 m_{i,i+1} m_{i+1,i+2}$ as $t\to\infty.$
Thus, by iteration, we derive
$$M_{ij}(t)\sim \tfrac{1}{(i-j)!}t^{j-i} m_{i,i+1}\cdots m_{j-1,j},\quad t\to\infty,$$
which allows us to conclude that condition \eqref{gammo} holds in the desired form, because
\[   \rE(\chi_j(t))\sim \tfrac{1}{(j-2)!}t^{j-2} m_{1,2}\cdots m_{j-1,j},\quad t\to\infty.
\]

Finally, to verify condition  (\ref{i})  with $\gamma=j-1$, it suffices to show that
\begin{equation}\label{final}
 \rE(\chi_j^2(t))\le Ct^{2j-3},\quad j=2,\ldots,q,
\end{equation}
using the following corollary of (\ref{aux})
\begin{align*}
 \rE(\chi_j^2(t+1))&=\rE\Big(\Big(\sum_{i=2}^j \sum_{k=1}^{N_{1i}}Z^{i}_t(j,k)\Big)^2\Big)\\
&=\sum_{i=2}^j m_{1i}V_{ij}(t)+2\sum_{i=2}^j \rE(N_{1i}(N_{1i}-1))M^2_{ij}(t)+2\sum_{2\le i<l\le j}\rE(N_{1i}N_{1l})M_{ij}(t)M_{lj}(t),
\end{align*}
where $V_{ij}(t):=\rE((Z^{i}_t(j))^2|Z^i_0=1)$.  From here, relation \eqref{final} is obtained from
\[\sum_{i=2}^j \rE(N_{1i}(N_{1i}-1))M^2_{ij}(t)+\sum_{2\le i<l\le j}\rE(N_{1i}N_{1l})M_{ij}(t)M_{lj}(t))\le C_1\sum_{2\le i\le l\le j} t^{j-i}t^{j-l}\le C_2t^{2j-4}\]
and the upper bound $V_{ij}(t)\le Ct^{2j-2i+1}$, derived next. 
Using \eqref{rau} and applying similar estimates, we find
\[V_{ij}(t+1)=\rE\Big(\Big(\sum_{l=i}^j \sum_{k=1}^{N_{il}}Z^{l}_t(j,k)\Big)^2\Big)\le\sum_{l=i}^j m_{il}V_{lj}(t)+Ct^{2j-2i}.\]
%\[\sum_{l=j+1}^i \sum_{k=j}^{l-1}\rE(N_{kl}N_{jl})M_{ki}(t)M_{li}(t)\le Ct^{2i-2j-1},\]
In particular, $V_{jj}(t+1)\le V_{jj}(t)+C$ 
implies $V_{jj}(t)\le Ct$. This in turn, gives
\[V_{j-1,j}(t+1)\le V_{j-1,j}(t)+C_1t+C_2t^2,\]
and $V_{j-1,j}(t)\le Ct^3$. Reiterating this argument, we find $V_{ij}(t)\le Ct^{2j-2i+1}$, which ends the proof of \eqref{final} and  Theorem  \ref{deco} as a whole.


\begin{thebibliography}{99pekin}\label{ref}
\bibitem{AN} Athrea, K. and Ney, P. Branching processes. John Wiley \& Sons, London-New York-Sydney, 1972.
%\bibitem{BGT} Bingham, N.H.; Goldie, C.M and Teugels, J.L. Regular variation. Cambridge university press, Cambridge, 1987.
%\bibitem{BK}  Bose, A. and Kaj, I. A scaling limit process for the age-reproduction structure in a Markov population.  Markov Process. Related Fields  6  (2000),  no. 3, 397--428.
%\bibitem{F51} Feller, W. Diffusion processes in genetics.  Proceedings of the Second Berkeley Symposium on Mathematical Statistics and Probability, 1950,  pp. 227--246. University of California Press, Berkeley and Los Angeles, 1951
\bibitem{F1} Feller, W. An introduction to probability theory and its applications. Vol. I. John Wiley \& Sons, New York-London-Sydney 1968
\bibitem{F2} Feller, W. An introduction to probability theory and its applications. Vol. II. John Wiley \& Sons, New York-London-Sydney 1971
%\bibitem{FS} Fleischmann, K. and Siegmund-Schultze, R. The structure of reduced critical Galton-Watson processes.  Math. Nachr.  79  (1977), 233--241
\bibitem{FN} Foster, J. and Ney, P. Limit laws for decomposable critical branching processes. Z. Wahrscheinlichkeitstheorie verw Gebiete 46  (1978) 13--43. 
%\bibitem{G} Goldstein, M. Critical age-dependent branching processes: Single and multitype.  Z. Wahrscheinlichkeitstheorie und Verw. Gebiete,  17  (1971) 74--88.
 \bibitem{G} Green, P. J. Conditional limit theorems for general branching processes.  J. Appl. Probability  14  (1977), no. 3, 451--463
\bibitem{HJV} Haccou, P., Jagers, P., and Vatutin, V.A. Branching processes: variation, growth, and extinction of populations. Cambridge university press, Cambridge, 2005
%\bibitem{H1} Holte, J.M. Extinction probability for a critical general branching process.  Stoch. Processes Appl. 2 (1974) 303-309.
\bibitem{H2} Holte, J.M. A generalization of Goldstein's comparison lemma and the exponential limit law in critical Crump-Mode-Jagers branching processes.  Adv. Appl. Prob. 8 (1976) 88--104.
%\bibitem{J69} Jagers, P. A general stochastic model for population development.  Skand. Aktuarietidskr.  1969, 84--103.
%\bibitem{J73} Jagers, P. On the asymptotics of general branching processes in the critical case. Intern. Conf. on Probability Theory and Mathematical Statistics. Vilnius, Abstracts of Communications 2 (1973) 361--364.
\bibitem{J74} Jagers, P. Convergence of general branching processes and functionals thereof.  J. Appl. Probability  11  (1974), 471--478.
\bibitem{J} Jagers, P. Branching processes with biological applications. John Wiley \& Sons, London-New York-Sydney, 1975.
%\bibitem{J82} Jagers, P. How probable is it to be first born? and other branching-process applications to kinship problems.  Math. Biosci.  59  (1982), no. 1, 1--15.
%\bibitem{J89} Jagers, P. General branching processes as Markov fields.  Stochastic Process. Appl.  32  (1989),  no. 2, 183--212.
%\bibitem{JN} Jagers, P. and Nerman, O. The growth and composition of branching populations.  Adv. in Appl. Probab.  16  (1984),  no. 2, 221--259
%\bibitem{JS} Jagers, P. and Sagitov, S.  The growth of general population-size-dependent branching processes year by year.  J. Appl. Probab.  37  (2000),  no. 1, 1--14.
\bibitem{JS1} Jagers, P. and Sagitov, S.  General branching processes in discrete time as random trees. Bernoulli 14 (2008) 949--962.
%\bibitem{1997} Kaj, I. and Sagitov, S. Superprocess approximation for a spatially homogeneous branching walk.  Electron. Comm. Probab.  2  (1997), 59--70 (electronic).
%\bibitem{1998} Kaj, I. and Sagitov, S. Limit processes for age-dependent branching particle systems.  J. Theoret. Probab.  11  (1998) 225--257.
%\bibitem{KW} Kawazu, K. and Watanabe, S. Branching processes with immigration and related limit theorems.  Theory Probab. Appl. 16 (1971) 36--54.
\bibitem{Kim} Kimmel, M. and Axelrod, D.E. Branching Processes in Biology, Springer-Verlag, New York, 2002.
\bibitem{L} Lamperti, J. The limit of a sequence of branching processes, Z. Wahrsch. Verw. Gebiete, 7 (1967), 271--288.
\bibitem{MM} M\"{o}hle, M. and Vetter B. Asymptotics of Continuous-Time Discrete State Space Branching Processes for Large Initial State. Markov Processes Relat. Fields 27 (2021) 1--42 
%\bibitem{N81} Nerman, O. On the convergence of supercritical general (C-M-J) branching processes.  Z. Wahrsch. Verw. Gebiete  57  (1981) 365--395. 
%\bibitem{N83} Nerman, O. A new proof of some asymptotical results for critical Crump-Mode-Jagers branching processes.  Preprint: Dept. Math., Chalmers University of Technology, 1983, 1--10.
%\bibitem{N84} Nerman, O. The stable pedigrees of critical branching populations.  J. Appl. Probab.  21  (1984) 447--463.
%\bibitem{NJ} Nerman, O. and Jagers, P. The stable double infinite pedigree process of supercritical branching populations.  Z. Wahrsch. Verw. Gebiete  65  (1984) 445--460
%\bibitem{83} Sagitov, S. Limit theorem for a critical branching process of general type. (Russian)  Mat. Zametki  34  (1983) 453--461. 
%\bibitem{86} Sagitov, S. Limit behavior of general branching processes. (Russian)  Mat. Zametki  39  (1986) 144--155, 159. 
\bibitem{90} Sagitov, S. A multidimensional critical branching process generated by a large number of particles of a single type. Theory Probab. Appl.  35  (1991) 118--130 
%\bibitem{94} Sagitov, S. Measure-branching renewal processes.  Stochastic Process. Appl.  52  (1994),  no. 2, 293--307
%\bibitem{95} Sagitov, S. Three limit theorems for reduced critical branching processes. Russian Math. Surveys  50  (1995),  no. 5, 1025--1043
%\bibitem{95b} Sagitov, S. A key limit theorem for critical branching processes. Stochastic Process. Appl. 56 (1995), no. 1, 87--100.
%\bibitem{97} Sagitov, S. Limit skeleton for critical Crump-Mode-Jagers branching processes.  Classical and modern branching processes (Minneapolis, MN, 1994),  295--303, IMA Vol. Math. Appl., 84, Springer, New York, 1997. 
%\bibitem{13} Sagitov, S.  Linear-fractional branching processes with countably many types. Stoch. Proc. Appl. 123 (2013) 2940-2956
\bibitem{Sev1} Sevast'yanov, B. A. Transient Phenomena in branching stochastic processes, Theory Probab. Appl., 4 (1959) 113–128
\bibitem{Seva} Sewastjanow, B.A. Verzweigungsprozesse, Akademie-Verlag, Berlin, 1974.
\bibitem{ZT} Taib, Z. Branching processes and neutral evolution. Lecture Notes in Biomathematics, 93. Springer-Verlag, Berlin, 1992
%\bibitem{T} Topchii, V. A. Properties of the probability of nonextinction of general critical branching processes under weak restrictions. Siberian Math. J. 28 (1987) 832--844.
%\bibitem{V77} Vatutin, V. A. Discrete limit distributions of the number of particles in critical Bellman-Harris branching processes. Theory Probab. Appl.    22  (1977) 146–152.
%\bibitem{V79} Vatutin, V. A. A new limit theorem for a critical Bellman-Harris branching process. (Russian)  Mat. Sb. (N.S.)  109(151)  (1979), no. 3, 440--452, 480. 

\bibitem{V86} Vatutin, V. A. Critical Bellman–Harris branching processes starting with a large number of particles. Math. Notes 40 (1986), 803–811.
\bibitem{V89} Vatutin, V. A. Asymptotic properties of Bellman-Harris critical branching processes starting with a large number of particles.
Journal of Soviet Mathematics  47, (1989) 2673–2681.
%\bibitem{V93} Vatutin, V. A. The total number of particles in a reduced Bellman-Harris branching process. Theory Probab. Appl.  38  (1993),  no. 3, 567--571 
%\bibitem{Z} Zubkov, A. M. Limit distributions of the distance to the nearest common ancestor. (Russian)  Teor. Verojatnost. i Primenen.  20  (1975), no. 3, 614--623
%\bibitem{Ya} Yakymiv, A. L. Multidimensional Tauberian theorems and their application to Bellman- Harris branching processes  Math. USSR Sb. 43  (1982), no. 3, 413--425
\end{thebibliography}
\end{document}